\documentclass[12pt, reqno]{amsart}
\usepackage{amssymb,latexsym,amsmath,amsfonts}
\usepackage[mathscr]{eucal}

\usepackage{color}
\definecolor{Red}{rgb}{1,0,0}
\definecolor{Blue}{rgb}{0,0,1}

\numberwithin{equation}{section}

\theoremstyle{definition}
\newtheorem{definition}{Definition}[section]

\theoremstyle{remark}

\theoremstyle{plain}
\newtheorem{theorem}[definition]{Theorem}

\newtheorem{lemma}[definition]{Lemma}
\newtheorem{proposition}[definition]{Proposition}

\newtheorem{corollary}[definition]{Corollary}

\newcommand{\ZZ}{\mathbb{Z}}

\newcommand{\rl}{\mathbb{R}}
\newcommand{\rd}{\mathbb{R}^d}

\begin{document}
\setlength{\abovedisplayskip}{3pt}
\setlength{\belowdisplayskip}{3pt}

\title[Existence of extremizers for a model convolution operator]{Existence of extremizers for a model convolution operator}
\author{Chandan Biswas}

\thanks{This work is supported in part by NSF grants DMS-1266336 and DMS-1600458}

\address{Department of Mathematics, University of Wisconsin, Madison, USA}
\email{cbiswas@wisc.edu}
\subjclass[1991]{Primary 42B10, 44A35, 44A12(secondary).}

\begin{abstract} 
The operator $T$, defined by convolution with the affine arc length measure on the moment curve 
parametrized by $h(t)=(t,t^{2},...,t^{d})$ is a bounded operator from 
$L^{p}$ to $L^{q}$ if $(\frac{1}{p}, \frac{1}{q})$ lies on a line segment.
In this article we prove that at non-end points there exist functions which extremize the associated inequality and 
any extremizing sequence is pre compact modulo the action of the symmetry of $T$. We also establish a relation between extremizers for $T$ at the end points and the extremizers of an X-ray transform restricted to directions along the moment curve. Our proof is based on the ideas of Michael Christ on convolution with the surface measure on the paraboloid. 
\end{abstract}

\maketitle
\tableofcontents
\section{Introduction and statement of results}\label{S:intro}
Let $X$ and $Y$ be two Banach spaces and $T$ be a linear operator from $X$ into $Y$. Beyond the immediate question of boundedness of $T$ or the value of the operator norm of $T$, it is natural to investigate the operator in more details, such as the various properties of the operator - both qualitative and quantitative - of the existence of extremizers, the properties of extremizers, the extremizing sequences, the near extremizers of the operator.

This kind of detailed study of a bounded operator has a long and rich history. One of the most celebrated examples is the work of William Beckner \cite{WB}, where he studied the existence of extremizers for the Hausdorff-Young inequality on $\mathbb{R}^d$. Recently there has been a series of works on the above questions for different operators, such as the Stein-Tomas Inequality \cite{TS,MC2}, an improved Hausdorff-Young inequality \cite{HY}, convolution with the surface measure on the paraboloid \cite{MC5,MC1}, for $k$-plane transforms \cite{Alex,TF}, convolution with the surface measure on the sphere \cite{BS2} to name a few. One motivation to investigate such questions is to produce improved inequalities and thus an inverse result concerning the stability of the inequality near an extremizer (see \cite{HY}, for example) and to make qualitative studies of PDEs \cite{W}.

In this paper we investigate the above questions for a generalized Radon transform along the moment curve. Let $T$ be the linear operator which acts on the continuous functions on $\mathbb{R}^d$ by convolution with affine arc length measure on the moment curve $(t, t^2,..., t^d)$. That is,
for a continuous function $f$ on $\mathbb{R}^d$, $Tf$ is defined by
\begin{equation}\label{defT}
Tf(x)= \int_{\mathbb{R}} f(x+(t, t^{2},..., t^{d})) dt.
\end{equation}
\begin{theorem}{\bf(Christ, Littman, Oberlin, Stovall)}\label{boundT}
$T$ maps $L^{p}(\mathbb{R}^d)$ into $L^{q}(\mathbb{R}^d)$ as a bounded operator with $1\leq p, q\leq\infty$ if and only if
\begin{equation}\label{pvalue}
\frac{1}{p}=\frac{1}{p_\theta}=\frac{1-\theta}{p_0}+\frac{\theta}{p_1}\qquad\text{and}\qquad\frac{1}{q}=\frac{1}{q_\theta}=\frac{1-\theta}{q_0}+\frac{\theta}{q_1}
\end{equation}
for some $\theta\in[0,1]$ where $p_{0}=\frac{d+1}{2}$, $q_{0}=\frac{d(d+1)}{2(d-1)}$ and $p_{1}=q_{0}', q_{1}=p_{0}'$. 
\end{theorem}

The above theorem was proved for $d=2$ by Littman in \cite{L}. Oberlin \cite{O} showed that it is sufficient to satisfy the above condition when $d=3$. The theorem was proved up to the end point for any dimension by Christ in \cite{MC4}. Extending ideas from \cite{MC4}, Stovall proved the strong type end point bound \cite{BS1}.

By using methods based on several observations and results of Christ (see \cite{MC4,MC5,MC3, BS1}), such as the method of refinement to get quantitative restricted weak type inequality and an almost orthogonality argument using a distance-like function between two near extremizers, together with some refinements by Dendrinos and Stovall (using the increasing structure in the method of refinement), \cite{BS3}, we have been able to improve the associated inequality for this operator for $\theta\in(0,1)$, such as establishing existence of functions that optimize the inequality and a qualitative information about functions that nearly optimize the corresponding inequality. At the endpoint i.e. for $\theta=0,1$ we apply a scaling, not adapted to the moment curve, to establish existence of functions that optimize the inequality corresponding to the bounds of an Xray transform. To state our results more precisely, first we need to introduce some definitions.

Let $(p, q)$ be as above. Let $A_p$ be the operator norm of $T$. That is 
$$
A_p= \sup _{\|f\|_{L^p}=1} \|Tf\|_{L^q}.
$$
Note that although $A_p$ depends on $p$, we shall write it simply as $A$ when it is clear what value of $p$ is under consideration.
\begin{definition}\textbf{Extremizer.}\label{ext}
Let $f\in L^p$. We say that $f$ is an extremizer if  
\begin{equation}\label{ext}
\|Tf\|_{L^q}= A \|f\|_{L^p}\neq 0.
\end{equation}
\end{definition}
\begin{definition}\textbf{$\delta$-Quasiextremizer.}
For any $\delta\textgreater 0$, $f\in L^p$ is a $\delta$-quasiextremizer  if
\begin{equation}\label{delta}
\|Tf\|_{L^q}\geq\delta  \|f\|_{L^p}\neq 0.
\end{equation}
\end{definition}
\begin{definition}\textbf{$\delta$-Quasiextremizer pair.}\label{QEP}
Let $\delta\textgreater 0$. We say that an ordered pair $(f,g)$ of measurable functions on 
$\mathbb{R}^d$ is an $\delta$-quasiextremizer pair if 
$$ 
\langle T(f), g \rangle \geq \delta  \|g\|_{L^{q'}} \|f\|_{L^{p}}\neq 0.
$$
\end{definition}
\begin{definition}\textbf{Extremizing sequence.}\label{Extsq}
An extremizing sequence is any sequence $f_n\in L^{p}$ such that
$$
\|f_n\|_{L^p}=1,
$$
$$
\|Tf_n\|_{L^q}\rightarrow A.
$$
\end{definition}

Let us denote the moment curve by $h(t)=(t,t^2,...,t^d)$. Our main theorems are as follows.
\begin{theorem}\label{MT}
For every $\theta\in(0,1)$ there exists an extremizer for $T:L^{p_{\theta}}\rightarrow L^{q_{\theta}}$ for the 
inequality (\ref{ext}) when $d\geq 2$. Furthermore for any nonnegative extremizing sequence $\{f_{n}\}$ 
there exists a sequence of symmetries (diffeomorphisms of $\rd$, preserving $L^p$ norm of $f$), $\{\phi^{*}_{n}\}$, of $\mathbb{R}^{d}$, 
such that there is a subsequence of $\{\phi^{*}_n(f_n)\}$ which converges in $L^{p_\theta}$ to some extremizer for 
$T:L^{p_{\theta}}\rightarrow L^{q_{\theta}}$.
\end{theorem} 

To state the second theorem we need a few more definitions.

Let $X$ be the X-ray transform restricted to directions along the moment curve defined on continuous functions on $\mathbb{R}^d=\mathbb{R}\times \mathbb{R}^{d-1}$ by
\begin{equation}\label{defX}
Xf(t,y)=\int_{\mathbb{R}} f(s,y+s(2t,3t^2,...,dt^{d-1})) \,ds.
\end{equation}

\begin{theorem}{\bf(Christ, Erdogan \cite{CE}, Dendrinos, Stovall \cite{BS4}, Erdogan \cite{E}, Laghi \cite{Laghi})}\label{CELSD}
$X$ maps $L^p$ into $L^{q}(L^{r}, dt)$ if for some $\theta\in[0,1)$
$$
\Big(\frac{1}{p}, \frac{1}{q}, \frac{1}{r}\Big)=\Big(\frac{1}{p_\theta}, \frac{1}{q_\theta}, \frac{1}{r_\theta}\Big)=\Big(1-\theta+\frac{\theta d}{d+2},\,\,\frac{\theta \,d}{d+2},\,\,1-\theta+\frac{\theta(d^2-d-2)}{d^2+d-2}\Big)
$$
and the restricted weak type bound holds for $X$ at the end point i.e., for $\theta=1$.
\end{theorem}

The X-ray transform has been studied by many authors for its connection to many other parts of mathematics. It was first studied by Gelfand in \cite{GG}. There has been a lot of work done to investigate the boundedness properties of the X-ray like transforms, such as \cite{GGV} and \cite{GS}, to name a few. 

It has been proved by Michael Christ, that extremizers of $T$ exist in the case $d=2$, and they have been identified and shown to be unique up to symmetries. Although it is still not known whether for $d>2$, there exists an extremizer for $T:L^{\frac{d+1}{2}}\rightarrow L^{\frac{d(d+1)}{2(d-1)}}$(corresponding to $\theta=0$ in Theorem~\ref{boundT}), we have been able to prove the following.
\begin{theorem}\label{endpoint}
Let $T$ be defined as in \ref{defT} and $d>2$.
\begin{itemize}\label{MT2}
\item Every extremizing sequence for $T:L^{p_0}\rightarrow L^{q_0}$ has a subsequence that either converges modulo symmetries of $T$ to an extremizer for $T:L^{p_0}\rightarrow L^{q_0}$, or that converges modulo the nonsymmetry, $f_n\rightarrow r^{\frac{2(d-1)}{d+1}}_nf_n((0,r_{n}x')+h(x_1))$, to an extremizer for $X^{*}:L^{p_0}\rightarrow L^{q_0}$ corresponding to $\theta=\frac{(d+2)(d-1)}{d^2+d}$ in Theorem~\ref{CELSD}.
\item Likewise, every extremizing sequence for $T:L^{p_1}\rightarrow L^{q_1}$ has a subsequence that either converges modulo symmetries of $T$ to an extremizer for $T:L^{p_1}\rightarrow L^{q_1}$, or that converges modulo nonsymmetry, $f_n\rightarrow r^{\frac{d^2-d+2}{d+1}}_nf_n(r_{n}x)$, to an extremizer for $X:L^{p_1}\rightarrow L^{q_1}$.
\item $\|T\|_{L^{p_0}\rightarrow L^{q_0}}\geq\|X^{*}\|_{L^{p_0}\rightarrow L^{q_0}}$ and if there exists an extremizing sequence for $T:L^{p_0}\rightarrow L^{q_0}$ that does not have a subsequence converging to an extremizer modulo symmetries of $T$, then $\|T\|_{L^{p_0}\rightarrow L^{q_0}}=\|X^{*}\|_{L^{p_0}\rightarrow L^{q_0}}$.
\end{itemize}
\end{theorem}

\begin{corollary}
At least one of the following must hold:
\begin{itemize}
\item(A) There exists an extremizer for $T:L^{p_0}\rightarrow L^{q_0}$; or
\item (B) There exists an extremizer for $X:L^{p_1}\rightarrow L^{q_1}$.
\end{itemize}
\end{corollary}

\section{Outline of the proof}
A simplified outline of the argument is as follows. Given a function $f$ which is very close to being an extremizer, we consider a dyadic decomposition of the range: $f=\sum_j 2^{j}f_j$ where $\frac{1}{2}\chi_{E_j}\leq f_j<\chi_{E_{j}}$ for pairwise disjoint measurable sets $E_j$ in $\mathbb{R}^d$. Following \cite{MC5} we prove that for $f$ to be a near extremizer (say $\|Tf\|_{q}\geq A(1-\delta)\|f\|_p$ for small $\delta$) the set of indices $\{j\}$ essentially lies in an interval around $J\in\ZZ$ in $\ZZ$ of length depending only on $\delta$ but independent of $f$. This is done by applying a certain``trilinear" bound for $T$ (see Lemma~\ref{extrapolation} and Theorem~\ref{highint}) using Christ's method of refinement \cite{MC4} and the increasing structure of the refinement due to Dendrinos and Stovall \cite{BS3}. The trilinear bound has already been established in Lemma $5.2$ in \cite{BS1}. Our proof is much simpler in comparison due to the increasing structure in the method of refinement. We show that $E_j$ are very close to being curved ``parallelepipeds", and these are ``almost" pairwise disjoint. More precisely these parallelepipeds are projections of balls in the incidence manifold. This is accomplished by introducing a mock distance on the set of all natural extremizers and proving that any two distant natural extremizers are almost orthogonal in the $L^q$ space, see Lemma~\ref{orthogonal} (this is similar to the argument due to Christ in \cite{MC5}). One significant difference in the structure of these parallelepipeds corresponding to when $d=2$
(in this case the paraboloid in \cite{MC5} is the same as the moment curve) from $d>2$ is when $d=2$ the symmetries of the operator act transitively on this set of parallelepipeds. For $d>2$ the situation is quite different. For $\theta\in(0,1)$ the symmetries of the operator act transitively on this set of parallelepipeds (near extremizers for $L^{p_\theta}\rightarrow L^{q_\theta}$ bound). So after applying symmetries we can assume that all these near extremizers are well adapted to the unit ball in $\rd$. In other words $E_J\sim B(0,1)$, the unit ball in $\rd$. So we can proceed as in the case of Paraboloid \cite{MC5} using Christ's argument to prove the existence of extremizers. On the other hand when $d>2$, the symmetry group does not act transitively on the set of near extremizers for $L^{p_0}\rightarrow L^{q_0}$ bound. As a consequence one has to allow the thickness of these parallelepipeds to become arbitrarily small as $f$ becomes closer to being an extremizer. We overcome this obstruction by applying a non-symmetric ``scaling" to the function $f$, so that we can avoid an extremizing sequence converging to $0$ pointwise while simultaneously preserving the $L^p$ norm of $f$. This is essentially the only new part of our analysis. This enables us to make the now rescaled near extremizers for $T:L^{p_0}\rightarrow L^{q_0}$ well adapted to the unit ball in $\rd$ and thus proving existence of extremizers for the Xray tranform, $X:L^{p_0}\rightarrow L^{q_0}$.
\smallskip

\section{Notation} 
Most of the notation we will use is fairly standard. In this note
$c, C$ denote implicit small and large positive constants respectively,
which are allowed to change from one line to another. If $1\leq p\leq\infty$, 
we denote by $p^{'}$ the exponent dual to $p$. We use $|E|$ to indicate 
Lebesgue measure. When A and B are non-negative real numbers, we write 
$A\lesssim B$ to mean $A\leq CB$ for an implicit constant $C$, and $A\sim B$ 
when $A\lesssim B$ and $B\lesssim A$. We will also employ the somewhat less 
standard notation $\mathcal{T}(E,F):=\langle T({\chi_{E}}), \chi_{F}\rangle$
when $E$ and $F$ are Borel sets and $T$ is a linear operator. We will also use 
$(E, F)$ to denote the pair of functions $(\chi_{E}, \chi_{F})$. We say that the sequence $\{f_{n}\}\subset L^{p}$ converges weakly to $f$ in $L^{p}$ if for any function $\psi\in L^{p'}$, $\int f_{n}\psi$ converges to
$\int f\psi$ and $\{f_{n}\}$ converges strongly to $f$ in $L^{p}$ if
$\int |f_{n}-f|^{p}$ converges to $0$. Since $T(f)\leq T(|f|)$ for all $f\in L^{p}$ and we are interested in only all those $f$ for which
$|T(f)|$ is large, in this paper all the functions $f$ will be assumed to be nonnegative.
\smallskip

\section{Symmetries}
In this section we study the symmetries of the operator $T$.

\begin{definition}
A symmetry of $T:L^{p}\rightarrow L^{q}$ is an $L^p$ isometry $\phi^{*}$ for which $T\circ\phi^{*}=\psi^{*}\circ T$ for some $L^q$ isometry $\psi^{*}$.
\end{definition}

The operator $T$ has many symmetries. Let $\Theta : \mathbb{R}^{d+d}\rightarrow \mathbb{R}^{d-1}$ 
be the function defined by
$$
\Theta (x,y)=(y_2-x_2-(y_1-x_1)^2, y_3-x_3-(y_1-x_1)^3,....., y_d-x_d-(y_1-x_1)^d)
$$
and Let $\Sigma$ be the incidence manifold $\Sigma=\{(x,y): \Theta(x,y)=0\}$. Let us denote 
the set of all diffeomorphisms of $\mathbb{R}^d$ by Diff$(\mathbb{R}^d)$.

\begin{definition}
Let $G_{d,d}$ denote the set of all $(\phi,\psi)\in\text{Diff}(\mathbb{R}^d)\times\text{Diff}(\mathbb{R}^d)$ such that 
$$
\Theta(\phi(x),\psi(y))= 0\quad \text{if and only if} \quad\Theta (x,y)=0 \quad\text{for all} \quad(x,y)\in\mathbb{R}^{d+d}.
$$
\end{definition}
In other words, $G_{d,d}$ denotes the set of all ordered pairs of diffeomorphisms of $\mathbb{R}^d$ which
preserve the incidence manifold $\Sigma$. We also let $G_d$ denote the set of all $\phi\in 
\text{Diff}(\mathbb{R}^d)$ such that there exists $\psi\in\text{Diff}(\mathbb{R}^d)$ such that 
$(\phi,\psi)\in G_{d,d}$.

The followings are examples of elements of $G_{d,d}$.
\begin{itemize}
\item Translation: $(\phi(x), \psi(y))=(x+v, y+v)$ for some $v\in \mathbb{R}^d$.
\item Scaling: $(\phi(x), \psi(y))=\big(S_r(x), S_r(y)\big)=\big((rx_1,r^2x_2,...,r^dx_d),(ry_1,r^2y_2,...,r^dy_d)\big)$
for some $r\in\mathbb{R} -\{0\}$.
\item\textit{Gliding along h}: $(\phi(x), \psi(y))= (G_{t_0}(x),G_{t_0}(y)+h(t_0))$ for some $t_0\in\mathbb{R}$,
where $G_{t_0}$ is the linear operator defined on $\mathbb{R}^d$ associated to the $(d\times d)$
matrix
\[
G_{t_0}=
  \begin{bmatrix}
    1 & 0 & 0\ldots & 0\\
     2t_0 & 1 & 0 \ldots & 0\\     
     3t_0^2 & 3t_0 & 1 \ldots & 0\\
     \vdots & \vdots & \ldots&0\\
     {m\choose 1}t_0^{m-1} & {m\choose 2}t_0^{m-2} & \ldots&0\\
     \vdots & \vdots & \ldots & 0\\
     {d\choose 1}t_0^{d-1} & {d\choose 2}t_0^{d-2} & \ldots & 1
     \end{bmatrix}.
\]
\end{itemize}
Note that $h(t+t_0)=G_{t_0}(h(t))+h(t_0)$ for all $t,t_0\in\mathbb{R}$.

The elements of $G_d$ play a central role in our analysis. There might be more elements in $G_d$ than the ones in the above examples but as
we shall see these are enough for our analysis. For each of the three types of symmetries 
described above the associated diffeomorphism has 
constant Jacobian. For each $\phi$ we define the associated operators 
$\phi^*:L^p\rightarrow L^p$ by $\phi^*f(x)=J_{\phi}^{\frac{1}{p}} f(\phi(x))$. Then

$$
\|\phi^*(f)\|_{L^{p}}=\|f\|_{L^{p}}, \quad
\langle T(\phi^{*}f),\psi^{*}g \rangle = \langle T(f),g \rangle .
$$
\smallskip

\section{Paraballs}

In this section we shall study an essentially exhaustive list of quasiextremal pairs. 
They are natural in the sense that every quasi-extremal pair is close, in a sense that degrades as the constant 
of quasiextremality decreases, to one of these pairs, see Theorem \ref{ballandquasi}. It is elementary to show that the 
characteristic function of the set $\{x\in\mathbb{R}^d:\|x\|<\delta\}$ is an $\delta^C$-quasiextremal for $T:L^{p_\theta}\rightarrow L^{q_\theta}$ for 
each $\theta\in [0,1]$. For $\theta=0$, we have, in addition, that the characteristic function of the set $\{x\in\mathbb{R}^d:\|x\|<\delta\}$, and for $\theta=1$,
the $\delta$-tubular neighborhood of the set $\{(t,t^2,....,t^d):t\in [-1,1]\}$ i.e. $\{x\in\mathbb{R}^d: \|x-(t,t^2,...,t^d)\|< \delta\,\, \text {for some}\,\, t\in[-1,1]\}$ are $c$-quasiextremal where $c$ is a small positive number that depends only on $d$ and independent of $\delta$, see Proposition \ref{FQE}. The set of all paraballs is the collection of sets that we produce by applying the elements of $G_d$ to these sets. Below is a more detailed description of the ``paraballs".

\begin{definition}For $0<\alpha, \beta\leq1$, we define
\begin{itemize}
\item $B(0,0,\alpha,1)=\{y\in\mathbb{R}^d:|y_1|<1\,\text{and}\,\,\|y-h(y_1)\|<\alpha\}$\,\,and\,\,$B^{*}(0,0,\alpha,1)=\{x\in\mathbb{R}^d:|x_i|<\alpha\,\,\text{for all}\,\,1\leq i\leq d\}$.
\item $B(0,0,1,\beta)=\{y\in\mathbb{R}^d:|y_i|<\beta\,\,\text{for all}\,\,1\leq i\leq d\}$\,\,and\,\,$B^{*}(0,0,1,\beta)=\{x\in\mathbb{R}^d:|x_1|<1 \,\text{and}\, \,\|x+h(-x_1)\|<\beta\}$.
\item 
\begin{equation}
\begin{split}
&B(\bar{x},t_0,\lambda\alpha,\lambda)=G_{t_0}S_{\lambda}B(0,0,\alpha,1)+\bar{x}+h(t_0)\\
&B^{*}(\bar{x},t_0,\lambda\alpha,\lambda)=G_{t_0}S_{\lambda}B^{*}(0,0,\alpha,1)+\bar{x}.
\end{split}
\end{equation}
\end{itemize}
\end{definition}
We also define a scaling of a paraball by
\begin{equation}\label{scale}
\lambda B(\bar x, t_0,\alpha,\beta)= B(\bar x, t_0, \lambda\alpha,\lambda\beta).
\end{equation}
Note that this does not correspond to a symmetry of the operator.

As an example in the special case when $0<\alpha\leq\beta$ the paraball $B=B(\bar x, t_0,\alpha,\beta)$ is the set of all 
$y\in\mathbb{R}^d$ satisfying all of 
\begin{itemize}
\item $|y_1-\bar y_1| \leq \beta$;
\item $|\sum_{i=1}^{m}\binom{m}{i}{(-t_0)}^{m-i}(y_i-\bar y_i)-(y_1-\bar y_1)^m|\leq \beta^{m-1}\alpha$\,\,\,\,for all\,\,\,\,$1< m\leq d$                                                 
\end{itemize}
where $\bar y=\bar x+h(t_0)$.

For $0<\alpha\leq\beta$, the dual paraball, denoted by $B^{*}=B^{*}(\bar x, t_0,\alpha,\beta)$, is
the set of all $x\in\mathbb{R}^d$ such that 
\begin{itemize}
\item $|x_1-\bar x_1| \leq\alpha$;
\item  $|\sum_{i=1}^{m}\binom{m}{i}{(-t_0)}^{m-i}(x_i-\bar x_i)|\leq \beta^{m-1}\alpha$\,\,\,\,for all\,\,\,$1< m\leq d$.                                           
\end{itemize}

Similarly when $0<\beta<\alpha$ the paraball $B=B(\bar x, t_0,\alpha,\beta)$ is the set of all 
$y\in\mathbb{R}^d$ satisfying all of 
\begin{itemize}
\item $|y_1-\bar y_1| \leq\beta$;
\item  $|\sum_{i=1}^{m}\binom{m}{i}{(-t_0)}^{m-i}(y_i-\bar y_i)|\leq \alpha^{m-1}\beta$\,\,\,\,for all\,\,\,$1< m\leq d$                                           
\end{itemize}
where $\bar y=\bar x+h(t_0)$.

The dual paraball, denoted by $B^{*}=B^{*}(\bar x, t_0,\alpha,\beta)$ is
the set of all $x\in\mathbb{R}^d$ such that
\begin{itemize}
\item $|x_1-\bar x_1| \leq \alpha$;
\item $|\sum_{i=1}^{m}\binom{m}{i}{(-t_0)}^{m-i}(x_i-\bar x_i)+(\bar x_1-x_1)^m|\leq\alpha^{m-1}\beta$\,\,\,\,for all\,\,\,\,$1< m\leq d$.                                                  
\end{itemize}

For our analysis of an extremizing sequence it is important to measure 
how two paraballs interact with each other. We define a mock-distance
on the set of all paraballs to measure the interaction between 
any two distant paraballs similar to the distance defined in \cite{MC5}.
\begin{definition}\label{distform}
Let $B^{a}=B(\bar x^{a}, t_{a},\alpha_{a},\beta_{a})$ and $B^{b}=B(\bar x^{b}, t_{b},\alpha_{b},\beta_{b})$ 
be two paraballs. Let $\bar y^a=\bar x^a+h(t_a)$ and $\bar y^b=\bar x^b+h(t_b)$ be the centers of the dual 
paraballs of $B^a$ and $B^b$ respectively. We define:
\begin{itemize}
\item If $\alpha_a\leq\beta_a$ and $\alpha_b\leq\beta_b$,
\begin{equation}
\begin{split}
d(B^a, B^b)&=\frac{\max\big(\alpha_{a}\beta_{a}\alpha_{a}{\beta_{a}}^{2}...\alpha_a\beta_a^{d-1}, \quad\alpha_b\beta_b\alpha_{b}{\beta_{b}}^{2}...\alpha_b\beta_b^{d-1}\big)}    {\min\big(\alpha_{a}\beta_{a}\alpha_{a}{\beta_{a}}^{2}...\alpha_a\beta_a^{d-1}, \quad\alpha_b\beta_b\alpha_{b}{\beta_{b}}^{2}...\alpha_b\beta_b^{d-1}\big)}+\Big(\frac{\beta_a}{\beta_b} + \frac{\beta_b}{\beta_a}\Big)\\ 
                 &+\Big(\frac{\alpha_a}{\alpha_b} + \frac{\alpha_b}{\alpha_a}\Big)
                 +|\bar y_1^a-\bar y_1^b|\bigg(\frac{1}{\beta_a}+\frac{1}{\beta_b}\bigg)
                 +|\bar x_1^a-\bar x_1^b|\bigg(\frac{1}{\alpha_a}+\frac{1}{\alpha_b}\bigg)\\
                 &+\sum_{m=2}^{d} \frac{|\sum_{i=1}^{m} \binom{m}{i}{(-t_{a})}^{m-i}(\bar y_{i}^b-
                 \bar y_{i}^{a})-(\bar y_{1}^b-\bar y_{1}^a)^m|}{\alpha_a\beta_{a}^{m-1}}\\
                 &+\sum_{m=2}^{d} \frac{|\sum_{i=1}^{m} \binom{m}{i}{(-t_{b})}^{m-i}(\bar y_{i}^a-
                 \bar y_{i}^{b})-(\bar y_{1}^a-\bar y_{1}^b)^m|}{\alpha_b\beta_{b}^{m-1}}\\
                 &+\sum_{m=2}^{d} \frac{|\sum_{i=1}^{m} \binom{m}{i}{(-t_{a})}^{m-i}(\bar x_{i}^b-\bar x_{i}^a)|}{\alpha_a\beta_{a}^{m-1}}\\
                 &+\sum_{m=2}^{d} \frac{|\sum_{i=1}^{m} \binom{m}{i}{(-t_{b})}^{m-i}(\bar x_{i}^a-\bar x_{i}^b)|}{\alpha_b\beta_{b}^{m-1}};
\end{split}
\end{equation}
\item If $\beta_a<\alpha_a$ and $\beta_b<\alpha_b$,
\begin{equation}
\begin{split}
d(B^a, B^b)&=\frac{\max\big(\alpha_{a}\beta_{a}\alpha_{a}^2{\beta_{a}}...{\alpha_a}^{d-1}\beta_a, \quad\alpha_b\beta_b\beta_{b}{\alpha_{b}}^{2}...\beta_b\alpha_b^{d-1}\big)}    {\min\big(\alpha_{a}\beta_{a}\beta_{a}{\alpha_{a}}^{2}...\beta_a\alpha_a^{d-1}, \quad\alpha_b\beta_b\beta_{b}{\alpha_{b}}^{2}...\beta_b\alpha_b^{d-1}\big)}+\Big(\frac{\beta_a}{\beta_b} + \frac{\beta_b}{\beta_a}\Big)\\ 
                 &+\Big(\frac{\alpha_a}{\alpha_b} + \frac{\alpha_b}{\alpha_a}\Big)
                 +|\bar y_1^a-\bar y_1^b|\bigg(\frac{1}{\beta_a}+\frac{1}{\beta_b}\bigg)
                 +|\bar x_1^a-\bar x_1^b|\bigg(\frac{1}{\alpha_a}+\frac{1}{\alpha_b}\bigg)\\
                 &+\sum_{m=2}^{d} \frac{|\sum_{i=1}^{m} \binom{m}{i}{(-t_{a})}^{m-i}(\bar y_{i}^b-\bar y_{i}^a)|}{\beta_a\alpha_{a}^{m-1}}\\
                 &+\sum_{m=2}^{d} \frac{|\sum_{i=1}^{m} \binom{m}{i}{(-t_{b})}^{m-i}(\bar y_{i}^a-\bar y_{i}^b)|}{\beta_b\alpha_{b}^{m-1}}\\
                 &+\sum_{m=2}^{d} \frac{|\sum_{i=1}^{m} \binom{m}{i}{(-t_{a})}^{m-i}(\bar x_{i}^b-
                 \bar x_{i}^{a})+(\bar x_{1}^a-\bar x_{1}^b)^m|}{\beta_a\alpha_{a}^{m-1}}\\
                 &+\sum_{m=2}^{d} \frac{|\sum_{i=1}^{m} \binom{m}{i}{(-t_{b})}^{m-i}(\bar x_{i}^a-
                 \bar x_{i}^{b})+(\bar x_{1}^b-\bar x_{1}^a)^m|}{\beta_b\alpha_{b}^{m-1}}.
\end{split}
\end{equation}
\end{itemize}
\end{definition}

A few comments are in order. Note that we have defined this distance $d$ for only two of the four possible cases, only for ($\alpha_a\leq\beta_a, \alpha_b\leq\beta_b$) and ($\beta_a<\alpha_a, \beta_b<\alpha_b$). This is because in our analysis we need to use the distance between only these two types of paraballs. So from now on when when we talk about distance between paraballs it would be one of these two cases. Note that $d$ is not a distance on the set of all paraballs, simply because for any paraball $B$, $d(B,B)=5$. But as we shall see that this is not of any significance to our analysis, for we shall use the properties of $d$ only when the distance between two paraballs is large. Note that our ``distance" function, $d$ is not a pseudo-distance either, as it does not satisfy the properties of a pseudo-distance, So for the lack of a better term we shall call it a mock-distance.

In the first term in the expression we compare the $(d-1)$-dimensional volume of the cross sections of the paraballs. The second and
the third terms measure the ratio between the lengths of the bases of the paraballs and the dual paraballs respectively. The fourth term measures the distance between the first coordinates of the centers of the paraballs and the fifth term for the centers of the dual paraballs. The sixth and the seventh term measure how far are the centers of each paraball from the other paraball. Likewise the eighth and ninth terms measure how far are the centers of the dual paraball from the other dual paraball.

We shall see in the proof of proposition \ref{distintersection} that the third, eighth and the ninth terms are redundant, in the sense that these are essentially dominated by first and second, sixth and the seventh term respectively. But we include these terms to make the mock-distance symmetric i.e. $d(B^a, B^b)=d({B^a}^{*}, {B^b}^{*})$. In addition we have the following property of this mock-distance.

\begin{lemma}
For every pair of paraballs $B^a, B^b$ and $\phi\in G_d$ we have
$$
d(B^a, B^b)=d(\phi^*(B^a), \phi^*(B^b)).
$$
\end{lemma}

\begin{proof}
It is enough to prove the stated equality when $\phi$ is either a translation or scaling or gliding along $h$. Let $B^{a}=B(\bar x^{a}, t_{a},\alpha_{a},\beta_{a})$ and $B^{b}=B(\bar x^{b}, t_{b},\alpha_{b},\beta_{b})$ be two paraballs. When $\phi$ is a translation i.e. $\phi(x)=x+v$ for some $v\in\rd$, then
$\phi^*(B^{a})=B(\bar x^{a}+v, t_{a},\alpha_{a},\beta_{a})$ and $\phi^*(B^{b})=B(\bar x^{b}+v, t_{b},\alpha_{b},\beta_{b})$. When $\phi$ is a scaling i.e. $\phi=S_r$ for some $r>0$, then $\phi^*(B^{a})=B(S_r(\bar x^{a}), rt_{a},r\alpha_{a},r\beta_{a})$ and $\phi^*(B^{b})=B(S_r(\bar x^{b}), rt_{b},r\alpha_{b},r\beta_{b})$. In both these cases each term in the Definition~\ref{distform} remains unchanged. Therefore $d(B^a, B^b)=d(\phi^*(B^a), \phi^*(B^b))$.

Let us now assume that $\phi$ is a gliding along $h$ i.e. $\phi(x)=G_{t_0}(x)$ for some $t_0\in\rl$. Then $\phi^*(B^{a})=B(G_{t_0}(\bar x^{a}), t_{a}+t_0,\alpha_{a},\beta_{a})$ and $\phi^*(B^{b})=B(G_{t_0}(\bar x^{b}), t_{b}+t_0,\alpha_{b},\beta_{b})$. So the first five terms in $d(\phi^*(B^a), \phi^*(B^b))$ remain same as the first five terms in $d(B^a),(B^b))$. If $\alpha_a\leq\beta_a$ then the sixth term in 
$d(\phi^*(B^a), \phi^*(B^b))$ is
\begin{equation}
\begin{split}
&\sum_{m=2}^{d}\frac{|\sum_{i=1}^{m} \binom{m}{i}{(-t_{a}-t_0)}^{m-i}(G_{t_0}(\bar y^b-\bar y^{a}))_i-(\bar y_{1}^b-\bar y_{1}^a)^m|}{\alpha_a\beta_{a}^{m-1}}\\
&=\sum_{m=2}^{d}\frac{|(G_{-t_0-t_a}G_{t_0}(\bar y^b-\bar y^{a}))_m-(\bar y_{1}^b-\bar y_{1}^a)^m|}{\alpha_a\beta_{a}^{m-1}}\\
&=\sum_{m=2}^{d}\frac{|(G_{-t_a}(\bar y^b-\bar y^{a}))_m-(\bar y_{1}^b-\bar y_{1}^a)^m|}{\alpha_a\beta_{a}^{m-1}}\\
&=\sum_{m=2}^{d} \frac{|\sum_{i=1}^{m} \binom{m}{i}{(-t_{a})}^{m-i}(\bar y_{i}^b-\bar y_{i}^{a})-(\bar y_{1}^b-\bar y_{1}^a)^m|}{\alpha_a\beta_{a}^{m-1}},
\end{split}
\end{equation}
which is the also sixth term in $d(B^a,B^b)$. Similarly when $\beta_a<\alpha_a$ the sixth term remains unchanged. By similar arguments one can prove that all other terms remain unchanged in $d(\phi^*(B^a), \phi^*(B^b))$. This finishes the proof.
\end{proof}

\begin{proposition}\label{distintersection}
There exists a constant $C<\infty$ which depends only on the dimension d, 
such that for any two paraballs $B^a,B^b$
$$
d(B^{a},B^{b})\leq C\Bigg(\frac{\max \big(|B^{a}|, |B^{b}|\big)}{|B^{a}\cap B^{b}|}\Bigg)^C.
$$
\end{proposition}
\begin{proof}
The proof of this lemma will be an adaptation of the proof of Lemma $3.7$ in \cite{MC5}.
We shall give the proof for the case $\alpha_a\leq\beta_a$ and $\alpha_b\leq\beta_b$ for the 
paraballs $B^{a}=B(\bar{x}^a,t_a,\alpha_a,\beta_a)$ and $B^{b}=B(\bar{x}^b,t_b,\alpha_b,\beta_b)$ respectively, the other case being identical. Without loss of generality we 
may assume that $d(B^a,B^b)$ is large. Otherwise $d(B^a,B^b)$ would be bounded by a 
large constant $C$. For any paraball $B(\bar{x}, t_{0},\alpha,\beta)$, $y\in B$ implies
$|y_{1}-\bar{y}_{1}|\leq\beta$ where $\bar{y}=\bar{x}+h(t_0)$. This implies that 
\begin{multline*}
|B^{a}\cap B^{b}|\leq \min (\alpha_{a}\beta_{a}\alpha_{a}{\beta_{a}}^{2}...\alpha_a\beta_a^{d-1},\,
\alpha_b\beta_b\alpha_{b}{\beta_{b}}^{2}...\alpha_b\beta_b^{d-1}) \text{min}(\beta_a, \beta_b)\\
\quad\leq \frac{\text{min} (\alpha_{a}\beta_{a}\alpha_{a}{\beta_{a}}^{2}...\alpha_a\beta_a^{d-1},\,
\alpha_b\beta_b\alpha_{b}{\beta_{b}}^{2}...\alpha_b\beta_b^{d-1})}{ \text{max} 
(\alpha_{a}\beta_{a}\alpha_{a}{\beta_{a}}^{2}...\alpha_a\beta_a^{d-1},\,
\alpha_b\beta_b\alpha_{b}{\beta_{b}}^{2}...\alpha_b\beta_b^{d-1})}\text{max}(|B^{a}|, |B^{b}|).\\
\end{multline*}
If $\frac{\max\big(\alpha_{a}\beta_{a}\alpha_{a}{\beta_{a}}^{2}...\alpha_a\beta_a^{d-1},\,
\alpha_b\beta_b\alpha_{b}{\beta_{b}}^{2}...\alpha_b\beta_b^{d-1}\big)}{\min\big(\alpha_{a}
\beta_{a}\alpha_{a}{\beta_{a}}^{2}...\alpha_a\beta_a^{d-1},\,\alpha_b\beta_b\alpha_{b}
{\beta_{b}}^{2}...\alpha_b\beta_b^{d-1}\big)}\geq c\,d(B^{a}, B^{b})$, this concludes the proof. 

For the paraball $B^a=B(\bar{x}^a,t_a,\alpha_a,\beta_a)$, we define
$$
S^a=(-\beta_a, \beta_a).
$$
Similarly, we define $S^b$ for $B^b$. Now
\begin{multline*}
\quad\quad\quad \quad\quad\quad\quad|B^{a}\cap B^{b}|\leq \frac{|\big(\bar{y}_{1}^{a}+S^{a}\big)\cap\big(\bar{y}_{1}^{b}+S^{b}\big)|}{\text{max}(|S^{a}|, |S^{b}|)} \text{max}(|B^{a}|, |B^{b}|)\\
                          \quad\quad\leq \frac{\text{min} (|S^{a}|, |S^{b}|)}{ \text{max} (|S^{a}|, |S^{b}|)} \text{max}(|B^{a}|, |B^{b}|)\\
                          =\frac{\text{min} (\beta_{a}, \beta_{b})}{ \text{max} (\beta_{a}, \beta_{b})} \text{max}(|B^{a}|, |B^{b}|)\\
                          \quad\sim\bigg(\frac{\beta_a}{\beta_b} + \frac{\beta_b}{\beta_a}\bigg)^{-1}\text{max}(|B^{a}|, |B^{b}|).\\
 \end{multline*} 
Therefore the desired inequality follows if $\frac{\beta_a}{\beta_b}+\frac{\beta_b}{\beta_a}\geq c\,d(B^{a}, B^{b})$.
In addition, for some absolute constant $C$,
$$
|\big(\bar{y}_{1}^{a}+S^{a}\big)\cap\big(\bar{y}_{1}^{b}+S^{b}\big)|\leq C \bigg[|\bar{y}_{1}^{a}-\bar{y}_{1}^{b}|
\bigg(\frac{1}{\beta_{a}}+\frac{1}{\beta_b}\bigg)\bigg]^{-1} \text{max}(|S^{a}|, |S^{b}|).
$$
Hence the desired inequality follows if the fourth term is $\geq c\,d(B^{a}, B^{b})$.

Let us now consider the third term. We can assume that the first two terms are small i.e.
\begin{itemize}
\item $\frac{\beta_a}{\beta_b} + \frac{\beta_b}{\beta_a}\leq c' d(B^a,B^b)^{c'}$;
\item $\frac{\max\big(\alpha_{a}\beta_{a}\alpha_{a}{\beta_{a}}^{2}...\alpha_a\beta_a^{d-1}, \quad\alpha_b\beta_b\alpha_{b}{\beta_{b}}^{2}...\alpha_b\beta_b^{d-1}\big)}    {\min\big(\alpha_{a}\beta_{a}\alpha_{a}{\beta_{a}}^{2}...\alpha_a\beta_a^{d-1}, \quad\alpha_b\beta_b\alpha_{b}{\beta_{b}}^{2}...\alpha_b\beta_b^{d-1}\big)}\leq c' d(B^a,B^b)^{c'}$.
\end{itemize}
We shall prove that this implies the third term is also small i.e.
$$
\frac{\alpha_a}{\alpha_b} + \frac{\alpha_b}{\alpha_a}\leq cd(B^a,B^b)^c.
$$
WLOG let us assume that $\beta_a\leq\beta_b$. Since $\frac{\beta_a}{\beta_b} + \frac{\beta_b}{\beta_a}\leq c' d(B^a,B^b)^{c'}$,
this implies 
$$
\beta_a\leq\beta_b\leq c'd(B^a,B^b)^{c'}\beta_a.
$$ 
Therefore
\begin{equation}
\begin{split}
\Big(\frac{\alpha_a}{\alpha_b} + \frac{\alpha_b}{\alpha_a}\Big)^{d-1}&\sim\frac{\text{max}({\alpha_a}^{d-1},{\alpha_b}^{d-1})}{\text{min}({\alpha_a}^{d-1},{\alpha_b}^{d-1})}\\
&\leq\frac{\max\big(\alpha_{a}\beta_{b}\alpha_{a}{\beta_{b}}^{2}...\alpha_a\beta_b^{d-1}, \quad\alpha_b\beta_b\alpha_{b}{\beta_{b}}^{2}...\alpha_b\beta_b^{d-1}\big)}    {\min\big(\alpha_{a}\beta_{a}\alpha_{a}{\beta_{a}}^{2}...\alpha_a\beta_a^{d-1}, \quad\alpha_b\beta_a\alpha_{a}{\beta_{a}}^{2}...\alpha_b{\beta_a}^{d-1}\big)}\\
&\sim\bigg(c'd(B^a,B^b)^{c'}\bigg)^{C}\frac{\max\big(\alpha_{a}\beta_{a}\alpha_{a}{\beta_{a}}^{2}...\alpha_a\beta_a^{d-1}, \quad\alpha_b\beta_b\alpha_{b}{\beta_{b}}^{2}...\alpha_b\beta_b^{d-1}\big)}    {\min\big(\alpha_{a}\beta_{a}\alpha_{a}{\beta_{a}}^{2}...\alpha_a\beta_a^{d-1}
\quad\alpha_b\beta_b\alpha_{b}{\beta_{b}}^{2}...\alpha_b\beta_b^{d-1}\big)}.
\end{split}
\end{equation}
We choose $c'$ such that $\bigg(c'd(B^a,B^b)^{c'}\bigg)^{C}<cd(B^a,B^b)^c$.

Now we consider the fifth term. Suppose that $|\bar{x}_{1}^{a}-\bar{x}_{1}^{b}|\geq c d(B^{a}, B^{b}) \alpha_{a}$.
If $(w, s_{1},..., s_{d-1})\in\mathbb{R}\times\mathbb{R}^{d-1}$ belongs to $B^{a}\cap B^{b}$, then one has
$|s_{1}-\bar{x}_{2}^{a}-(w-\bar{x}_{1}^{a})^{2}|\leq \alpha_{a}\beta_{a}$ and 
$|s_{1}-\bar{x}_{2}^{b}-(w-\bar{x}_{1}^{b})^{2}|\leq \alpha_{b}\beta_{b}$. Subtracting gives us
$$
|2w\, (\bar{x}_{1}^{a}-\bar{x}_{1}^{b})-d|\leq 2 \max (\alpha_{a}\beta_{a}, \alpha_{b}\beta_{b}),
$$
where $d=2(\bar{x}_{1}^{a})^{2}-2\bar{x}_{1}^{a}\bar{y}_{1}^{b}$. Since $|\bar x^a_1-\bar x^b_1|\geq c d(B^a,B^b)\alpha_a$,
this implies
$$
\big|\{w\in S^{a}: |2w(\bar{x}_{1}^{a}-\bar{x}_{1}^{b})-d|\leq 2 \max (\alpha_{a}\beta_{a}, \alpha_{b}\beta_{b}) \}\big|\leq
C d(B^{a}, B^{b})^{-1}|S^{a}|
$$
uniformly for all $d\in\mathbb{R}$. This implies the required upper bound on $|B^{a}\cap B^{b}|$.

Next let us assume that for some $m$ with $2\leq m\leq d$, we have
$$
\frac{|\sum_{i=1}^{m} \binom{m}{i}{(-t_{a})}^{m-i}(\bar y_{i}^b-\bar y_{i}^{a})-(\bar y_{1}^b- \bar y_{1}^a)|}{\alpha_a\beta_{a}^{m-1}}\geq c d(B^{a}, B^{b}).
$$ 
We define polynomials $Q^{a}_{j}$ and $Q^{b}_{j}$ on $\mathbb{R}^{d}$ by
\begin{multline*}
\quad\qquad\quad\qquad\qquad Q^{a}_{j}(y)=\sum_{i=1}^{j} \binom{j}{i}{(-t_{a})}^{j-i}(y_{i}-\bar y_{i}^{a})-(y_{1}-\bar y_{1}^a)^j\\
Q^{b}_{j}(y)=\sum_{i=1}^{j} \binom{j}{i}{(-t_{b})}^{j-i}(y_{i}-\bar y_{i}^{b})-(y_{1}-\bar y_{1}^b)^j\\
\end{multline*}
for each $2\leq j\leq d$.
For every $z\in\mathbb{R}$ we define $t(z)\in\mathbb{R}^{d-1}$ so that $Q^{b}_{j}(z, t(z))=0$ for each $j$. Now we
define a one variable polynomial $P(z)=Q^{a}_{m}(z, t(z))$.
Observe that 
$$
|P(z)|>\,\alpha_{a}\beta_{a}^{m-1}+\alpha_{b}\beta_{b}^{m-1}\quad\text{implies}\quad B^{a}\cap B^{b}\cap 
(\{z\}\times \mathbb{R}^{d-1})=\emptyset.
$$
Note that $P(\bar x_1^{b})\geq cd(B^{a}, B^{b}) \alpha_{a}\beta_{a}^{m-1}$. Let 
$$
\epsilon=\frac{3\max (\alpha_a\beta_{a}^{m-1}, \alpha_{b}\beta_{b}^{m-1})}{d(B^{a}, B^{b})\, \alpha_{a}\beta_{a}^{m-1}}\leq 
d(B^{a}, B^{b})^{-\frac{1}{2}}.
$$
Then for all $z\in \bar x_1^b+S^{b}$ we have
$$
P(z)\geq \epsilon d(B^{a}, B^{b}) \alpha_{a}\beta_{a}^{m-1}= 3\max (\alpha_a\beta_{a}^{m-1}, \alpha_{b}\beta_{b}^{m-1})
\geq \alpha_{a}\beta_{a}^{m-1}+\alpha_{b}\beta_{b}^{m-1}
$$
except on a set of measure smaller than $C\epsilon^{c}|S^{b}|$. Therefore one has
$$
|B^{a}\cap B^{b}|\leq C\epsilon^{c}|B^{b}|\leq C d(B^{a}, B^{b})^{-c} |B^{b}|.
$$
Similar arguments give the required inequality if $\frac{\big|\sum_{i=1}^m\binom{m}{i}{(-t_{b})}^{m-i}(\bar y_{i}^a-\bar y_{i}^{b})-(\bar y_{1}^a-\bar y_{1}^b)^m\big|}{\alpha_b\beta_{b}^{m-1}}\geq cd(B^a,B^b)$.

Now let us assume that for some $2\leq m\leq d$ we have 
$$
 \frac{|\sum_{i=1}^{m} \binom{m}{i}{(-t_{a})}^{m-i}(\bar x_{i}^b-\bar x_{i}^a)|}{\alpha_a\beta_{a}^{m-1}}\geq cd(B^a,B^b)
$$
and that all the previous terms are less than $c'd(B^a,B^b)$ where $c'$ is a small positive number to be chosen precisely in a moment. 
Since both sides of the equation are invariant if we replace $(B^a,B^b)$ by $(\phi^{*}(B^a), \phi^{*}(B^b))$, we can assume that $B^a=B(0,0,\alpha_a,1)$ and $B^b=B(\bar x^b,t,\alpha_b,\beta_b)$. This implies
\begin{equation}\label{dualin}
|\bar x_m^b|=|\bar y_m^b-t^m|> cd(B^a,B^b)\alpha_a
\end{equation}
and 
\begin{itemize}
\item $\frac{\alpha_a}{\alpha_b}+\frac{\alpha_b}{\alpha_a}< c'd(B^a,B^b)$;
\item $|\bar y_1^b-t|<c'd(B^a,B^b)\text{max}(\alpha_a,\alpha_b)$;
\item $|\bar y_m^b-(\bar{y_1}^b)^m|<c'd(B^a,B^b)\alpha_a$.
\end{itemize}

This implies $t=\bar y_1^b+\mathcal{O}(c'd(B^a,B^b))\alpha_a$ which in turn implies
$$
|\bar x_m^b|=|\bar y_m^b-t^m|=|\bar y_m^b-(\bar y_1^b)^m+\mathcal{O}(c'd(B^a,B^b))\alpha_a|<\mathcal{O}(c'd(B^a,B^b))\alpha_a.
$$
We choose $c'$ small enough such that this contradicts \ref{dualin}.
\end{proof}

We also have an almost triangle inequality.
\begin{lemma}\label{triangle}
There exists a constant $C<\infty$ depending only on the dimension $d$ 
such that for any three paraballs $B^{a},B^{b},B^{d}$ we have 
$$
d(B^{a},B^{b})\leq C \big(d(B^{a}, B^{d})^{C} + d(B^{d},B^{b})^{C}\big).
$$
\end{lemma} 

\begin{proof}
Without loss of generality we may assume that $B^{d}=B(0,0,\alpha,1)$ or $B(0,0,1,\beta)$.
The parameters specifying $B^{a}$ and $B^{b}$ are controlled by $\eta_{1}=d(B^{a}, B^d)$ and 
$\eta_{2}=d(B^{b}, B^d)$ respectively. Also for any three positive numbers $\beta_{1}, \beta_{2}, \beta$, 
one has $\frac{\beta_{1}}{\beta_{2}}+\frac{\beta_{2}}{\beta_{1}}\leq C \bigg(\big(\frac{\beta_{1}}{\beta}+\frac{\beta}{\beta_{1}}\big)^{C}+
\big(\frac{\beta_{2}}{\beta}+\frac{\beta}{\beta_{2}}\big)^{C}\bigg)$. Therefore $d(B^{a}, B^{b})$ is bounded from 
above by $C\eta^{C}$ where $\eta=\max \big(d(B^{a}, B^{d}),\,d(B^{d},B^{b})\big)$.
\end{proof}

We also have the following covering property.
\begin{lemma}
There exists a constant $C<\infty$ depending only on the dimension $d$ 
such that for any two paraballs $B^{a},B^{b}$ we have 
$$
B^{a}\subset C(d(B^{a},B^{b}))^C B^{b}
$$
\end{lemma}
where $C(d(B^{a},B^{b}))^C B^{b}$ is defined as in (\ref{scale}).
\begin{proof}
Without loss of generality we may assume that $B^{b}=B(0,0,\alpha,1)$. Let $B^{a}=B(\bar{x}, t_{0}, \alpha_1, \beta_1)$ and let 
$\eta=d(B^{a}, B^{b})$. Then there is a constant $C$ such that the parameters corresponding to $B^{a}$ are controlled by
$C\eta^C$. After some elementary algebra it follows that $B^{a}\subset C\eta^C B^{b}$.
\end{proof}

In the next Proposition we shall prove that $(B(0,0,1,1), B^{*}(0,0,1,1))$ (and hence $(B(0,0,\alpha,\alpha), B^{*}(0,0,\alpha,\alpha))$) are quasiextremal pairs for $T:L^{p_\theta}\rightarrow L^{q_\theta}$ for every $\theta\in [0,1]$. In Lemma \ref{compa} we shall prove that for $0<\theta<1$ these are essentially the only quasi-extremal pairs. For $\theta=0$, in addition to the above we also have 
$\big(B(0,0,\alpha,1),B^{*}(0,0,\alpha,1)\big)$ for every $0<\alpha<1$ and for $\theta=1$, we have $\big(B(0,0,1,\beta),B^{*}(0,0,1,\beta)\big)$ for every $0<\beta<1$ which are quasiextremal pairs for $T$.

\begin{proposition}{\label{FQE}}
There exists $c>0$ which depends only on the dimension $d$ with the following property. 
\begin{itemize}
\item $\big(B(0,0,\alpha,1), B^{*}(0,0,\alpha,1)\big)$ is a $c$-quasi-extremal pair for 
$T:L^{p_0}\rightarrow L^{q_0}$ for all $0<\alpha<1$;
\item $\big(B(0,0,1,\beta), B^{*}(0,0,1,\beta)\big)$ is a $c$-quasi-extremal pair for 
$T:L^{p_1}\rightarrow L^{q_1}$ for all $0<\beta<1$;
\item $\big(B(0,0,\alpha,\alpha), B^{*}(0,0,\alpha,\alpha)\big)$ is a $c$-quasi-extremal pair for 
$T:L^{p_\theta}\rightarrow L^{q_\theta}$ for all $0\leq\theta\leq 1$ and for all $\alpha>0$.
\end{itemize}
\end{proposition}

\begin{proof}
We shall write the proof of the first claim, the others being identical.
Let $B=B(0, 0,\alpha,1)$ with $0<\alpha<1$. We claim that
\begin{equation}\label{incident}
T(B, B^{*})\geq c\alpha^{d},\qquad |B|\leq \alpha^{d-1},\qquad
|B^{*}|\leq \alpha^{d}
\end{equation}
which after some elementary calculations implies that
$$
\frac{T(B, B^{*})}{|B|^{\frac{2}{d+1}}|B^{*}|^{1-\frac{2(d-1)}{d(d+1)}}}\geq c.
$$

The upper bounds on the sizes of $B$ and $B^{*}$ follow directly from the definition.
Let us fix a small number $r> 0$ (to be chosen precisely later) which depends only on $d$. Define $B_{r}$ to 
be the set of all $y\in\mathbb{R}^d$ such that
\begin{itemize}
\item $|y_{1}| \leq r$;
\item $|y_{m}-y_{1}^{m}|\leq r\alpha$ for all $1< m\leq d$.                                               
\end{itemize}
Then $|B_{r}|\geq r^d |B|$.

We want to show that if $r$ is sufficiently small then for all $y\in B_{r}$, the set of all 
$x\in\mathbb{R}^d$ such that $x\in B^{*}$ has measure at 
least $r\alpha$. Therefore $T(B, B^{*})\geq r^{d+1}\alpha |B|\geq r^{d+1}
|B|^{\frac{2}{d+1}}|B^{*}|^{1-\frac{2(d-1)}{d(d+1)}}$.

Let us fix $y\in B_r$. For each $x_{1}\in\mathbb{R}$ with $|x_{1}|< r\alpha$, we define
$x'=(x_2,x_3,...,x_d)\in\mathbb{R}^{d-1}$ by $x_{m}=y_{m}-(y_{1}-x_{1})^{m}$ for $2\leq m\leq d$, so that $(x, y)\in\Sigma$.
Now $x\in B^{*}$ if and only if 
\begin{equation}
|x_{m}|= |y_{m}-(y_{1}-x_{1})^{m}|<\alpha.
\end{equation}

Now,
\begin{equation}
\begin{split}
x_{m}&=y_{m}-(y_{1}-x_{1})^{m}\\
&=y_{1}^{m}-(y_{1}-x_{1})^{m}+\mathcal{O}{(r)}\alpha\\
&\leq\mathcal{O}{(r)}\alpha<\alpha,\\
\end{split}
\end{equation}
if we choose $r$ to be sufficiently small.
\end{proof}

\section{Quasiextremal pairs and paraballs}\label{chap3}
Let $E$ and $F$ be subsets of  $\mathbb{R}^d$ with finite positive Lebesgue measure. 
Write $\mathcal{T}(E,F)=\langle T(\chi_{E}), \chi_{F} \rangle$ and $\mathcal{T}(f,g)=\langle T(f), g\rangle$. Define
$\alpha$ and $\beta$ by 
 $$
 \alpha |E|= \beta |F|=  \langle T(\chi_{E}), \chi_{F} \rangle.
 $$
 
Then $T$ being restricted weak type $(p_0,q_0)=\big(\frac{d+1}{2},\frac{d(d+1)}{2(d-1)}\big)$ is equivalent to
$$
|E|\geq c \beta^{\frac{d(d+1)}{2}} {\bigg(\frac{\alpha}{\beta}\bigg)}^{d-1}.
$$
In addition, if $(E,F)$ is an $\epsilon$-quasiextremal pair then by definition \ref{QEP} we also have 
$$\label{upperbound}
|E|\leq c {\epsilon}^{-\textbf C} \beta^{\frac{d(d+1)}{2}} {\bigg(\frac{\alpha}{\beta}\bigg)}^{d-1}.
$$
for some $C> 0$. We aim to exploit these two inequalities simultaneously to obtain information about 
$\epsilon$-quasiextremal pairs and prove the following theorem.
\begin{theorem}\label{ballandquasi}
Let $d>2$. There exists an absolute constant $C$, depending only on $d$ such that for any $\epsilon$-quasiextremal 
pair $(E,F)$, there exists a paraball $B$ such that
$$
T(E\cap B, F\cap B^*)\geq {C}^{-1} {\epsilon}^{C} \mathcal{T}(E,F)
$$
and
$$
|B|\leq |E|\quad \text{and}\quad |B^*|\leq |F|. 
$$
\end{theorem}

\section{Parametrization of subsets of E and F}
The following Lemma is proved in Lemma $3.7$ in \cite{BS3}. 
\begin{lemma}\label{EL1} 
If $d$ is even there exists a point $\bar y$ in $E$, a measurable subset $\Omega\subset \mathbb{R}^{d+1}$ 
such that  

\begin{itemize}
\item $|\Omega|=c \alpha^{\frac{d+2}{2}} \beta^{\frac{d}{2}}$;
\item $\bar y -h(t_1)+h(t_2)-h(t_3)-...+h(t_j)\in E$ for every $t=(t_1,....,t_{d+1})\in\Omega$ and for every even $j$;
\item $\bar x-h(t_1)+h(t_2)-h(t_3)-...-h(t_j)\in F$ for every $t=(t_1,....,t_{d+1})\in\Omega$ and for every odd $j$;
\item $t_1< t_2<....< t_d$ for every $t=(t_1,t_2,...,t_{d+1})\in\Omega$;
\item $t_i-t_{i-1}\geq c\beta$ for every even $i$;
\item $t_i-t_{i-1}\geq c\alpha$ for every odd $i$;
\end{itemize}

and if $d$ is odd, there exists a point $\bar{x}$ in $F$, a measurable subset $\Omega\subset \mathbb{R}^{d+1}$ 
such that  
\begin{itemize}
\item $|\Omega|=c \alpha^{\frac{d+1}{2}} \beta^{\frac{d+1}{2}}$;
\item  $\bar{x}+h(t_{1})-h(t_{2})+h(t_{3})-...+h(t_{j})\in E$ for every $t=(t_{1},..,t_{d+1})\in\Omega$ and for every odd $j$;
\item $\bar{x}+h(t_{1})-h(t_{2})+h(t_{3})-...-h(t_{j})\in F$ for every $t=(t_{1},....,t_{d+1})\in\Omega$ and for every even $j$;
\item  $t_{1}< t_{2}<....< t_{d}$ for every $t=(t_{1}, t_{2},..., t_{d+1})\in\Omega$.
\item $t_i-t_{i-1}\geq c\beta$ for every odd $i$;
\item $t_i-t_{i-1}\geq c\alpha$ for every even $i$.
\end{itemize}
Here $c$ is a small positive constant independent of $E,F,\alpha,\beta$ and depends only on $d$.
\end{lemma}

Now as in \cite{MC4} if we consider the map $(t_1,...,t_d)$ goes to $\bar y -h(t_1)+h(t_2)-h(t_3)-...+(-1)^dh(t_d)\in E$ 
we have 
$$
|E|\geq c \int_{\Omega} \prod_{1\leq i < j\leq d} (t_j-t_i).
$$
Since $(E,F)$ is a $\epsilon$-quasiextremal in addition to Lemma \ref{EL1} we also have for each $t\in\Omega$, 
\begin{itemize}
\item $t_i \leq t_{i-1}+\epsilon^{-C}\alpha$ for every odd $1< i \leq d+1$;
\item $t_i \leq t_{i-1}+\epsilon^{-C}\beta$ for every even $1< i \leq d+1$;
\end{itemize}
where $C$ is an absolute constant depending only on $d$.

\begin{lemma}\label{EL2}
There exists $C<\infty$, depending only on $d$, with the following properties. If $(E,F)$ is an $\epsilon$-quasiextremal with 
$\alpha |E|=\beta |F|=\mathcal{T}(E, F)$ then there exists $t_0\in\mathbb{R}$ and a point $\bar y$ in $E$, 
a measurable subset $\Omega\subset \mathbb{R}^{d+1}$ such that if $d$ is even
\begin{itemize}
\item $|\Omega|=c \alpha^{\frac{d+2}{2}} \beta^{\frac{d}{2}}$;
\item  $\bar y-h(t_{1})+h(t_{2})-h(t_{3})+...+h(t_{j})\in E$ for every $t=(t_{1},....,t_{d+1})\in\Omega$ and for every even $j$;
\item $\bar y-h(t_{1})+h(t_{2})-h(t_{3})-...-h(t_{j})\in F$ for every $t=(t_{1},....,t_{d+1})\in\Omega$ and for every odd $j$;
\item $t_i \leq t_{i-1}+\epsilon^{-\textbf C}\alpha$ for every odd $1< i < d$;
\item $t_i \leq t_{i-1}+\epsilon^{-\textbf C}\beta$ for every even $1< i \leq d$;
\item  $t_{1}< t_2<....< t_d$ for every $t\in\Omega$;
\item $|t_1-t_0|\leq \epsilon^{-\textbf{C}}\alpha$ for all $t\in\Omega$:
\item $t_i-t_{i-1}\geq c\beta$ for every even $i$;
\item $t_i-t_{i-1}\geq c\alpha$ for every odd $i$.
\end{itemize}
and if $d$ is odd, there exists a point $\bar{x}$ in $F$, a measurable 
subset $\Omega\subset \mathbb{R}^{d+1}$ such that  
\begin{itemize}
\item $|\Omega|=c \alpha^{\frac{d+1}{2}} \beta^{\frac{d+1}{2}}$;
\item  $\bar{x}+h(t_{1})-h(t_{2})-h(t_{3})-...+h(t_{j})\in E$ for every $t=(t_{1},..,t_{d+1})\in\Omega$ and for every odd $j$;
\item $\bar{x}+h(t_{1})-h(t_{2})+h(t_{3})-...-h(t_{j})\in F$ for every $t=(t_{1},....,t_{d+1})\in\Omega$ and for every even $j$;
\item  $t_{1}< t_{2}<....< t_{d}$ for every $t=(t_{1}, t_{2},..., t_{d+1})\in\Omega$.
\item $|t_{1}-t_{0}|\leq \epsilon^{-C}\beta$ for all $t\in\Omega$:
\item $t_i-t_{i-1}\geq c\beta$ for every odd $i$;
\item $t_i-t_{i-1}\geq c\alpha$ for every even $i$.
\end{itemize}
\end{lemma}

\begin{proof}
The proof is quite straightforward. If the $\Omega$ from Lemma \ref{EL1} does not satisfy the 
property that for some $t_0$,  $|t_1-t_0|\leq \epsilon^{-C}\alpha$ for all $t\in\Omega$, 
then in the proof of Lemma \ref{EL1} we iterate the construction of the sets $\Omega_k$ upto (d+3) times. Now we fix a point $s\in\Omega_{d+3}$ and apply Lemma \ref{EL1} with $\bar y$ replaced by $\bar y-h(s_1)+h(s_2)$.
\end{proof}

\begin{lemma}\label{compa}
There exist $c,C<\infty$ with the following properties. Let $(E,F)$ be an $\epsilon$-quasiextremal pair for $T:L^{p_\theta}\rightarrow L^{q_\theta}$ and 
$\alpha|E|=\beta|F|=\mathcal{T}(E,F)$. Then
\begin{itemize}
\item If $\theta=0$ then $\alpha\leq C\epsilon^{-C}\beta$;
\item If $\theta=1$ then $\beta\leq C\epsilon^{-C}\alpha$;
\item If $0<\theta<1$ then $c\epsilon^{\frac{C}{1-\theta}}\alpha\leq\beta\leq C\epsilon^{\frac{-C}{\theta}}\alpha$.
\end{itemize}
\end{lemma}
\begin{proof}
We shall first consider the case when $\theta=0$ and $d$ is even. The proof when $d$ is odd is identical. If $(E,F)$ is an $\epsilon$-quasiextremal for $T:L^{p_0}\rightarrow L^{q_0}$, by Lemma \ref{EL2} one has for all $t=(t_1, t_2,...,t_d)\in\Omega$,
$$
t_2+c\alpha<t_3<t_1+C\epsilon^{-C}\beta<t_2+C\epsilon^{-C}\beta.
$$
This implies $\alpha<C\epsilon^{-C}\beta$.

Next Let us consider the case when $\theta=1$. If $(E,F)$ is an $\epsilon$-quasiextremal for $T:L^{p_1}\rightarrow L^{q_1}$. Then one has
$$
\langle T^{*}(\chi_F), \chi_E\rangle=\langle T(\chi_E), \chi_F\rangle\geq \epsilon |E|^{\frac{1}{p_1}}|F|^{\frac{1}{q'_1}}
$$
and
$$
\beta|F|=\alpha|E|=\langle T^{*}(\chi_F), \chi_E)\rangle.
$$
This implies $(F,E)$ is an $\epsilon$-quasiextremal pair of $T^{*}:L^{q'_1}=L^{p_0}\rightarrow L^{p'_1}=L^{q_0}$. Since $T^{*}$ is the convolution
with the affine arclength measure of $-h(t)$, one has 
$$
\beta\leq C\epsilon^{-C}\alpha.
$$

Let us now fix a $\theta\in (0,1)$. Let $(E,F)$ be such that
$$
\langle T(\chi_E), \chi_F\rangle < \epsilon^\frac{C}{1-\theta} |E|^{\frac{1}{p_0}}|F|^{\frac{1}{q'_0}}.
$$
This implies
\begin{equation}
\begin{split}
\langle T(\chi_E), \chi_F\rangle&=(\langle T(\chi_E), \chi_F\rangle)^{1-\theta} (\langle T(\chi_E), \chi_F\rangle)^\theta\\
&<\epsilon^{C}\big(|E|^{\frac{1}{p_0}}|F|^{\frac{1}{q'_0}}\big)^{1-\theta} (A_{p_1}|E|^{\frac{1}{p_1}}|F|^{\frac{1}{q'_1}}\big)^\theta\\
&<\epsilon^{C} A_{p_1}^\theta |E|^{\frac{1}{p_\theta}} |F|^{\frac{1}{q'_\theta}}.
\end{split}
\end{equation}
Similarly $\langle T(\chi_E), \chi_F\rangle<\epsilon^{C} A_{p_0}^{1-\theta} |E|^{\frac{1}{p_\theta}} |F|^{\frac{1}{q'_\theta}}$ if
$\langle T(\chi_E), \chi_F\rangle < \epsilon^\frac{C}{\theta} |E|^{\frac{1}{p_1}}|F|^{\frac{1}{q'_1}}$. This implies if $(E,F)$ is an
$\epsilon$-quasiextremal pair of $T:L^{p_\theta}\rightarrow L^{q_\theta}$ then $(E,F)$ is an $\epsilon^\frac{C}{1-\theta}$-quasiextremal pair of
$T:L^{p_0}\rightarrow L^{q_0}$ and $\epsilon^{\frac{C}{\theta}}$-quasiextremal pair of $T:L^{p_1}\rightarrow L^{q_1}$. Thus by the results for $\theta=0$ and
$\theta=1$, we have
$$
c\epsilon^{\frac{C}{1-\theta}}\alpha\leq\beta\leq C\epsilon^{\frac{-C}{\theta}}\alpha.
$$
\end{proof}

Let us consider the paraball $B=B(\bar{x}=\bar{y}-h(t_0), t_{0}, C \epsilon^{-C}\alpha, C\epsilon^{-C}\beta)$.
\begin{lemma} 
If $B$ is as above then if $d$ is even,
$$
E\cap B\supset \bar{y}-h(t_{1})+h(t_{2})-h(t_{3})-...+h(t_{d}),
$$
$$
F\cap B^{*}\supset\bar{y}-h(t_{1})+h(t_{2})+h(t_{3})-...-h(t_{d+1})
$$
for every $t\in\Omega$, and when $d$ is odd 
$$
E\cap B\supset \bar{y}- h(t_{1})+h(t_{2})-h(t_{3})-...+h(t_{d+1}),
$$
$$
F\cap B^{*}\supset\bar{y}-h(t_{1})+h(t_{2})-h(t_{3})-...-h(t_{d})
$$
for every $t\in\Omega$.
\end{lemma}
\begin{proof}
We will give the details when $d$ is even, $\epsilon=1$ and $\theta=0$; the proof for other cases are essentially the same.
By Lemma \ref{EL2} it is enough to prove that $\bar{y}-h(t_{1})+h(t_{2})+h(t_{3})-...+h(t_{d})\in B$ 
and $\bar{y}-h(t_{1})+h(t_{2})-h(t_{3})-...+h(t_d)-h(t_{d+1})\in B^{*}$ for every $t\in\Omega$. 

Let us first prove that $\bar{y}-h(t_{1})+h(t_{2})+h(t_{3})-...+h(t_{d})\in B$.
We can assume, by applying suitable symmetry, if necessary, that $\bar{x}=\bar{y}=0$ with $t_{0}=0$.
If $y=(y_1,y_2,...,y_d)=-h(t_1)+h(t_2)-h(t_3)-...+h(t_d)$ then Lemma \ref{EL2} implies
$$
|y_{1}|=|t_{1}-t_{2}+...-t_{d}|\leq C\beta,
$$
and
\begin{equation}
\begin{split}
y_{m}-y_{1}^{m}&=\big(-t_{1}^{m}+t_{2}^{m}-...+t_{d}^{m}\big)-(-t_{1}+t_{2}-...+t_{d})^{m}\\
&=-t_{1}^m+(t_{2}^m-t_{3}^m)+...+(t_{d-2}^m-t_{d-1}^m)+t_{d}^m-\big[-t_{1}+(t_{2}-t_3)+...+(t_{d-2}-t_{d-1})+t_{d}\big]^{m}\\
&=-t_{1}^{m}-(-t_1)^m+[t_{2}^m-t_{3}^m-(t_2-t_3)^m]+...+[t_{d-2}^{m}-t_{d-1}^{m}-(t_{d-2}-t_{d-1})^{m}]\\
&-\sum_{r_{1}+...+r_{\frac{d}{2}}=m, r_{i}\neq m} (-t_1)^{r_1}(t_{2}-t_{3})^{r_{2}}...(t_{d-2}-t_{d-1})^{r_{\frac{d}{2}-1}}t_{d}^{\frac{d}{2}}\\
&\leq C\alpha^m+C\beta^{m-1}|t_2-t_3|+C\beta^{m-1}|t_4-t_5|+...+C\beta^{m-1}|t_{d-2}-t_{d-1}|.
\end{split}
\end{equation}
By Lemma \ref{EL2} we have $|t_i-t_{i+1}|<C\alpha$\,\,\text{for all even}\,\,$2\leq i\leq d$.
By Lemma \ref{compa} we have $\alpha\leq C\beta$. Therefore we have $|y_m-y_{1}^m|<C\alpha\beta^{m-1}$. This proves our claim.

Now we shall prove that $\bar{y}-h(t_{1})+h(t_{2})-h(t_{3})-...+h(t_d)-h(t_{d+1})\in B^{*}$ for every $t\in\Omega$.
WLOG $\bar{x}=\bar{y}=0$ with $t_{0}=0$. Let $x=(x_1,x_2,...,x_d)=-h(t_{1})+h(t_{2})-h(t_{3})-...+h(t_d)-h(t_{d+1})$.
Then $x\in B^{*}$ if and only if $|x_1|<C\alpha$ and $|x_m|<C\alpha\beta^{m-1}$ for all $1<m\leq d$. By Lemma \ref{EL2}
$$
|x_{1}|=|-t_{1}+(t_{2}-t_3)+...+(t_{d}-t_{d+1})|\leq C\alpha
$$
and 
\begin{equation}
\begin{split}
|x_{m}|&=|-t_{1}^{m}+t_{2}^{m}-...+t_{d}^{m}-t_{d+1}^m|\\
&=|-t_{1}^m+(t_{2}^m-t_{3}^m)+...+(t_{d}^m-t_{d+1}^m)|\\
&\leq |-t_{1}^{m}|+C\beta^{m-1}|t_{2}-t_{3}|+...+C\beta^{m-1}|t_{d}-t_{d+1}|\\
&\leq C\beta^{m-1}\alpha.
\end{split}
\end{equation}

\end{proof}

Therefore we have $|E\cap B(\bar x, t_0,C\epsilon^{-C}\alpha,C\epsilon^{-C}\beta)|\geq |\{\bar y-h(t_1)+h(t_2)+h(t_3)-...+h(t_d):t\in\Omega\}|\geq c\epsilon^C|E|$. Similarly $|F\cap B^{*}(\bar x, t_0,C\epsilon^{-C}\alpha,C\epsilon^{-C}\beta)|\geq |\{\bar y-h(t_1)+h(t_2)-h(t_3)-...-h(t_{d+1}):t\in\Omega\}|\geq c\epsilon^C|F|$. This implies
 \begin{itemize}
 \item $\mathcal{T}(E\cap B, F\cap B^*)\geq c\epsilon^C\mathcal{T}(E,F)$
 \item$ |B|\leq C {\epsilon}^{-C} |E|$ 
 \item $ |B^*|\leq C {\epsilon}^{-C} |F| $
 \end{itemize}

This is stronger than the conclusion of Theorem \ref{ballandquasi} which will be proved in the following lemma.
\begin{lemma}\label{partition}
There exist absolute constants $N, C<\infty$ with the following property. For each paraball 
$B$ and given any $0<\delta\leq 1$, there exists a family of paraballs $\{B_l:l\in L\}$ with the
following properties,
\begin{itemize}
\item $B\subset\cup_{l\in L}B_l$;
\item $B^{*}\subset\cup_{l\in L}B_l^{*}$;
\item $|L|\leq N \delta^{-C}$;
\item $|B_l|\sim\delta |B|$ for all $l$;
\item $|B^{*}_l|\sim\delta |B^{*}|$ for all $l$.
\end{itemize}
\end{lemma}

\begin{proof}
The proof of this lemma will be similar to the proof of Lemma $7.2$ in \cite{MC3}. WLOG we can assume that $B=B(0,0,\alpha,1)$ or $B(0,0,1,\beta)$. Let $B=B(0,0,\alpha,1)$, the proof for $B(0,0,1,\beta)$ follows from a similar argument. Then $|B|\sim\alpha^{d-1}$ and $|B^*|\sim\alpha^{d}$. Let $\eta=\delta^{\frac{2}{d(d+1)}}$. Let us select a maximal $\eta^d\alpha$-separated subset of $B^{*}(0,0,\alpha,1)$ with respect to the regular Euclidean distance. Let us denote this set by $\{z^l:l\in L\}$. Then $|L|\leq C\eta^{-C}$.
Now we choose a maximal $\eta$-seperated of $[-1,1]$. Let us denote this set by $\{t_k:k\in K\}$. Then $|K|\leq C\eta^{-C}$.

Now we define $B_{l,k}=B(z^l,t_k,C\eta\alpha,C\eta)$. Then $|B_{l,k}|\sim\eta^{\frac{d(d+1)}{2}}\alpha^{d-1}\sim\delta |B|$ and $|B^{*}_{l,k}|\sim\eta^{\frac{d(d+1)}{2}}\alpha^d\sim\delta|B^*|$. By the definition of $B_{l,k}$ it directly follows that
$$
B^{*}\subset\cup_{l,k}B_{l,k}^{*}.
$$

We shall now prove that $B\subset\cup_{l,k}B_{l,k}$. It is enough to prove that
\begin{equation}\label{cover}
B(0,0,\eta^d \alpha,1)\subset\cup_{k} B(0,t_k,C\eta\alpha, C\alpha)
\end{equation}
since this implies
$$
B(0,0,1,1)=\cup_{l}B(z^l,0,\eta^d\alpha,1)\subset\cup_{l}\cup_{k} B(z^l,t_k,C\eta\alpha, C\alpha).
$$

To prove \ref{cover}, let $y\in B(0,0,\eta^d\alpha,1)$. Now choose $t_k$ such that $t_k\leq y_1<t_{k+1}$. We claim that $y\in B(0,t_k,C\eta\alpha, C\eta)$. Our claim is true if and only if for each $2\leq m\leq d$
$$
\bigg|\sum_{i=1}^m  \binom{m}{i}(-t_k)^{m-i}(y_i-t_k^i)-(y_1-t_k)^m\bigg|\leq C^m\eta^m\alpha.
$$

Now
$$
\bigg|\sum_{i=1}^m \binom{m}{i} (-t_k)^{m-i}(y_i-t_k^i)-(y_1-t_k)^m\bigg|=\bigg|\sum_{i=1}^m \binom{m}{i} (-t_k)^{m-i}(y_i-y_1^i)\bigg|\leq C\eta^d\alpha\leq C\eta^m\alpha
$$
as $|t_k|\leq 1$ and $y\in B(0,0,\eta^d\alpha,1)$.

\end{proof}

To complete the proof of the Theorem \ref{ballandquasi} we choose $C$ sufficiently large such that with $\delta=\epsilon^C$
we apply Lemma \ref{partition} to obtain paraballs $\{B_l\}_{l\in L}$ such that for each $l\in L$, $|B_l|\leq|E|$
and $|B^{*}_l|\leq|F|$. We now have 

\begin{equation}
T(E\cap B, F\cap B^{*})\\
\leq \sum_{l\in L}T(E\cap B_l, F\cap B^{*}_l).
\end{equation}

In addition we have $|L|\leq C\epsilon^{-C}$. Therefore there exists a paraball $B_l$ such that
$$
T(E\cap B_l, F\cap B^{*}_l)\geq C^{-1}\epsilon^{C} T(E\cap B, F\cap B^{*}).
\geq c \epsilon^{C} \mathbb{T}(E, F)
$$
and 
$$
|B_l|\leq |E|\,\,\,\,\,\,\,\,\,\,|B^{*}_l|\leq |F|.
$$

\section{Lorentz spaces and $\epsilon$-quasiextremal function}
\begin{definition}
Let $f$ be a nonnegative function which is finite almost everywhere. By a rough level set 
decomposition of $f$ we mean a representation of $f$ as 
$f =\sum_{j=-\infty}^{\infty}2^{j} f_j$ where $\chi_{E_j}\leq f_j\leq 2\chi_{E_j}$ with the sets 
$E_j$ pairwise disjoint and measurable.

We may approximate the Lorentz norms of $f$ by,
$$
\|f\|_{p,r}\sim
\left\{
	\begin{array}{ll}
		\bigg(\sum_{j}(2^{j} |E_{j}|^{\frac{1}{p}})^{r}\bigg)^{\frac{1}{r}},  & \mbox{if}\quad r<\infty \\
		\sup_{j}2^{j} |E_{j}|^{\frac{1}{p}}, & \mbox{if} \quad r=\infty
	\end{array}
\right.
$$
where $f =\sum_{j=-\infty}^{\infty}2^{j} f_j, f_j \sim \chi_{E_j}$ is a rough level set decomposition of $f$.
In particular, $L^{p, r}\big(\mathbb{R}^{d}\big)=\{f: \|f\|_{p,r}\,<\,\infty\}$.
\end{definition}

The following lemma is Theorem $4.1$ in \cite{BS1}.
\begin{lemma}\label{lorentzbound}
T maps $L^{p, r}$ boundedly to $L^{q}$ for every $r\in (p, q)$ for every $(p,q)$ as in \ref{pvalue}.
\end{lemma}

The following lemma is also proved in the proof of Theorem $4.1$ in \cite{BS1}.
\begin{lemma}\label{uniformsmall}
There exist $C, c> 0$ with the following property. Let $\epsilon> 0$. 
Let $f=\sum_{j} 2^{j}f_j, f_j\sim\chi_{E_{j}}$ and  $g=\sum_{k} 2^{k}g_k, g_k\sim\chi_{F_{k}}$ 
be such that either $\mathcal{T}(E_{j}, F_{k})\leq \epsilon |E_{j}|^{\frac{1}{p}} |F_{k}|^{1-\frac{1}{q}}$ or
$2^{j}|E_{j}|^{\frac{1}{p}}\leq \epsilon \|f\|_{p}$ for each $j$ and $k$. Then 
$\mathcal{T}(f, g)\leq C\, \epsilon^{c}\|f\|_{p}\,\|g\|_{q^{'}}$.
\end{lemma}

\begin{lemma}\label{quasi}
There exist $c, C<\infty$ with the following property. For each $\epsilon> 0$, if
$f$ is a nonnegative function with rough level set decomposition $f =\sum_{j=-\infty}^{\infty}2^{j}f_j, f_j\sim\chi_{E_{j}}$
and if $f$ is a $\epsilon$-quasiextremal then there exists $j\in\mathbb{Z}$ and a paraball $B$ such that
$$
\|2^{j}\chi_{E_{j}\cap B}\|_{p}\geq c\epsilon^{C}\|f\|_{p}
$$
and
$$
|B|\leq |E_{j}|.
$$
\end{lemma}
\begin{proof}
The proof of this lemma follows directly from Theorem \ref{ballandquasi} and the previous lemma.
If $f =\sum_{j}2^{j}f_j, f_j\sim\chi_{E_{j}}$ by the previous lemma there exists $j\in\mathbb{Z}$ such that
$\|2^{j}\chi_{E_{j}\cap B}\|_{p}\geq c\epsilon^{C}\|f\|_{p}$ and $E_{j}$ is an 
$\epsilon^{c}$-quasietremal. Now we apply Theorem (\ref{ballandquasi}) to $E_{j}$ to get the
desired conclusion.
\end{proof}
\smallskip

\section{Two key Lemmas}
In this section we shall prove two lemmas that will be used in the 
later sections. The first lemma is about how paraballs interact with 
each other when they are distant from each other. It shows in some sense
when we have a collection of paraballs which are at a large distance 
from each other, then their image under $T$ act on nearly disjoint portions of any given set. The 
precise statement is given below.

\begin{lemma}\label{orthogonal}
Let $d> 2$ and let $(\frac{1}{p}, \frac{1}{q})$ be on the line segment joining the points 
$(\frac{2}{d+1}, \frac{2(d-1)}{d(d+1)})$ and $(1-\frac{2(d-1)}{d(d+1)}, 1-\frac{2}{d+1})$. 
Then there exists a positive finite constant $C$ depending only on $d$
with the following property. Let $\{B_{i}\}_{i\in S}$ be a collection of paraballs such that
for any $i\neq j$ with $i,j\in S$ we have $d(B_i,B_j)\geq C \eta^{-C}$ for some $\eta > 0$.
Then for any $F$ subset of $\mathbb{R}^d$ with positive finite Lebesgue measure, we can write
$F=\sqcup F_{i}$ so that
$$
T(B_i, F_j)\leq \eta |B_i|^{\frac{1}{p}} |F_j|^{\frac{1}{q'}},   \text {for all}\quad i\neq j.
$$
\end{lemma}

\begin{proof}
The proof of this lemma will be a straightforward adaptation of the proof of Lemma $4.1$ in \cite{MC5}.
For the sake of completeness we give a sketch of the proof here. Define
$$
\gamma_i=\frac{1}{3}\eta  |F|^{\frac{1}{q'}-1} |B_i|^\frac{1}{p}
$$
and 
\begin{equation}\label{big}
\tilde F_i=\{x\in F: T(B_i)> \gamma_i \}.
\end{equation}

We note that 
\begin{equation}\label{smaller}
T(B_i, F\setminus \tilde F_i)\leq \gamma_i |F|\leq \frac{1}{3}\eta  |B_i|^{\frac{1}{p}} |F|^\frac{1}{q'}.
\end{equation}

Now choose $F_i\subset \tilde F_i$ such that $\cup_i \tilde F_i=\sqcup F_{i}$. Note that the there are many choices of
$F_i$. We just choose one such collection. Also, there might be elements in $F$ which do not belong to $\tilde{F}_i$ for all $i\in S$. We 
pick one $F_i$ and include these points to this particular set.
Since by (\ref{smaller}) we have $T(B_i, F\setminus \sqcup F_i)\leq  \frac{1}{3}\eta  |B_i|^{\frac{1}{p}} |F|^\frac{1}{q'}$,
it is enough to prove that  for $i\neq j$
\begin{equation}\label{smallest}
T(B_i, F_j)\leq  \frac{2}{3}\eta  |B_i|^{\frac{1}{p}} |F_j|^\frac{1}{q'}.
\end{equation}
Suppose (\ref{smallest}) does not hold. Then there exists $i\neq j$ such that
$T(B_i, F_j)> \frac{2}{3}\eta  |B_i|^{\frac{1}{p}} |F_j|^\frac{1}{q'}$.
For the rest of this proof we fix these two indices $i,j$.
Define $\mathcal{F}=F_j\cap \tilde F_i$. By (\ref{smaller}) $T(B_i, F_j\setminus\tilde F_i)
\leq \frac{1}{3}\eta  |B_i|^{\frac{1}{p}} |F|^\frac{1}{q'}$, so we have
\begin{equation}
\frac{1}{3}\eta  |B_i|^{\frac{1}{p}} |F|^\frac{1}{q'}\leq \mathcal{T}(B_i,\mathcal{F})\leq A |B_i|^{\frac{1}{p}}|\mathcal{F}|^{\frac{1}{q'}}.
\end{equation}
This implies
$$
|\mathcal{F}|\geq \bigg(\frac{1}{3}\bigg)^{q'} \eta^{q'}A^{-q'} |F|.
$$
Now we apply Theorem \ref{ballandquasi} to the pair $(B_j,\mathcal{F})$ to obtain 
a paraball $\tilde B_j$ such that 
\begin{equation}
|\tilde B_j|\leq |B_j|, \quad |\tilde B_j^*|\leq |\mathcal{F}|\leq |F|,\quad\quad
|\tilde B_j\cap B_j|\geq c \eta^\gamma |B_j|,\quad |\tilde B_j^*\cap \mathcal{F}|\geq c \eta^\gamma |F|.
\end{equation}

Now we replace $\mathcal{F}$ by $\tilde{\mathcal{F}}=\mathcal{F}\cap\tilde{B_j}^*$.
Since $|\tilde{\mathcal{F}}|\geq c \eta^\gamma |F|$ and $T(B_i)(x)> \gamma_i$ for all
$x\in F_i\supset \mathcal{F}\supset \tilde{\mathcal{F}}$, consequently
$$
T(B_i,\tilde{\mathcal{F}})\geq \gamma_i |\tilde{\mathcal{F}}|\geq c\eta^{\gamma} |\tilde{\mathcal{F}}|^{\frac{1}{q'}} |B_i|^{\frac{1}{p}}.
$$
This means the pair $(B_i, \tilde{\mathcal{F}})$ is a $c\eta^{\gamma}$-quasietxremal.
Therefore by applying Theorem \ref{ballandquasi} once more we get another paraball
$\tilde{B_i}$ such that
\begin{equation}
|\tilde B_i|\leq |B_i|, \quad |\tilde B_i^*|\leq |\tilde{\mathcal{F}}|\leq |F|,\quad\quad
|\tilde B_i\cap B_i|\geq c \eta^\gamma |B_i|,\quad |\tilde B_i^*\cap \tilde{\mathcal{F}}|\geq c \eta^\gamma |F|.
\end{equation}
Since $\tilde{B_i}^*\cap\tilde{B_j}^*\supset  \tilde{B_i}^*\cap\tilde{B_j}^*\cap \mathcal{F}\supset \tilde{B_i}^*\cap\tilde{\mathcal{F}}$,
we have 
$$
|\tilde{B_i}^*\cap\tilde{B_j}^*|\geq | \tilde{B_i}^*\cap\tilde{\mathcal{F}}|\geq c\eta^{\gamma} |F|\geq c \eta^{\gamma} \text{max}(|\tilde{B_i}^*|, |\tilde{B_j}^*|).
$$

Now by applying Proposition~\ref{distintersection} to the pair of dual paraballs $(\tilde{B_i}^*, \tilde{B_j}^*)$
we get $d(\tilde{B_i}^*, \tilde{B_j}^*)\leq C \eta^{-C}$. This implies
$$
d(\tilde{B_i}, \tilde{B_j})\leq C \eta^{-C}.
$$
Since $|\tilde{B_i}|\leq |B_i|$ and $|\tilde{B_i}\cap B_i|\geq c \eta^{\gamma} |B_i|$, we have 
$$
d(\tilde{B_i}, B_i)\leq C \eta^{-C}.
$$
Similarly 
$$
d(\tilde{B_j}, B_j)\leq C \eta^{-C}.
$$
By applying Lemma \ref{triangle}  we get $d(B_i, B_j)\leq C \eta^{-C}$, which contradicts our hypothesis.
\end{proof}

\begin{lemma}\label{extrapolation}
Let $d>2$ and $(\frac{1}{p}, \frac{1}{q})$ be a point on the line segment joining the points 
$(\frac{2}{d+1}, \frac{2(d-1)}{d(d+1)})$ and $(1-\frac{2(d-1)}{d(d+1)}, 1-\frac{2}{d+1})$.
There exists $C,C^{'}$ positive finite constants depending only on $d$ with the
following property. Let $E_1,E_2, F$ be subsets of $\mathbb{R}^d$ with positive finite Lebesgue
measure such that $T(\chi_{E_1})\geq \eta |E_1|^{\frac{1}{p}} |F|^{{\frac{1}{q^{'}}}-1}$  and 
$T(\chi_{E_2})\geq \eta |E_2|^{\frac{1}{p}} |F|^{{{\frac{1}{q^{'}}}-1}}$ on $F$, then if $|E_2|\geq |E_1|$ we have
$|E_2|\leq C^{'} \eta^{-C} |E_1|$.
\end{lemma}

\begin{proof}
This lemma  is essentially proved in the proof of Theorem $4.1$ in \cite{BS1}
by applying extrapolation method of Christ. Here we give a simplified proof
using the increasing structure, $(t_1< t_2<...<t_d)$ of $\Omega$ in Lemma \ref{EL2}. We shall give
the proof when $d$ is even for the other case being similar.

Let $p_0=\frac{d+1}{2}$ and $q_0=\frac{d(d+1)}{2(d-1)}$. Let us first consider the case
when $p=p_0$ and $q=q_0$. Define
$$
\alpha=\eta |E_1|^{\frac{1}{p_0}-1} |F|^{\frac{1}{{q_0}^{'}}},\quad \beta=\eta |E_1|^{\frac{1}{p_0}} |F|^{\frac{1}{{q_0}^{'}}-1} 
\text{and} \quad \gamma=\eta |E_2|^{\frac{1}{p_0}} |F|^{\frac{1}{{q_0}^{'}}-1}.
$$

Since $T(\chi_{E_1})\geq \eta |E_1|^{\frac{1}{p_0}} |F|^{\frac{1}{{q_0}^{'}}-1}$ on $F$,
we have 
$$
\langle \chi_{E_1}, T^{*}(\chi_{F})\rangle=T(E_1,F)\geq \eta |E_1|^{\frac{1}{p_0}} |F|^{\frac{1}{{q_0}^{'}}}.
$$
Therefore on a large subset of $E_1$, $T^{*}(\chi_{F})\geq\alpha$. Similarly
on a large subset of $F$, $T(\chi_{E_1})\geq \beta$ and $T(\chi_{E_2})\geq \gamma$.

Similar to the proof of Theorem \ref{ballandquasi} there exists
a point $\bar{y}\in E_1$ such that we can travel along the curve 
shifted to $\bar{y}$ inside $F$ for a length of $\alpha$. Then for each of 
these points on this travelled path we can travel back inside
$E_1$ for a length of $\beta$. We continue this process $d-1$ times. 
At the $d$th step we move into $E_2$ along the curve for a length $\gamma$.

As a result we get a  $\Omega\subset\mathbb{R}^{d}$ such that

\begin{itemize}
\item $|\Omega|=c \alpha^{\frac{d}{2}} \beta^{\frac{d}{2}-1} \gamma$;
\item  $\bar y - h(t_1)+h(t_2)-h(t_3)-...+h(t_j)\in E_1$ for every $t=(t_1,....,t_d)\in\Omega$ and for every even $j\leq d-2$;
\item $\bar y - h(t_1)+h(t_2)-h(t_3)-...-h(t_j)\in F$ for every $t=(t_1,....,t_d)\in\Omega$ and for every odd $j\leq d-1$;
\item $\bar y - h(t_1)+h(t_2)-h(t_3)-...+h(t_d)\in E_2$;
\item  $t_1< t_2<....< t_d$ for every $t=(t_1,t_2,...,t_d)\in\Omega$.
\end{itemize}

Now we consider the Jacobian, $J(t)$, of the map $(t_1,t_2,...,t_d)\mapsto \bar x + h(t_1)-h(t_2)+h(t_3)-...-h(t_d)$.
We have 
$$
J\geq \prod_{1\leq i< j\leq d}|t_i-t_j|.
$$
Since $t_1< t_2< ...< t_d$ for every $t\in\Omega$, we have
\begin{itemize}
\item $\prod_{1\leq i< d}(t_i-t_d)\geq \gamma^{d-1}$;
\item $\prod_{1\leq i< j}(t_i-t_j) \geq \beta^{j-1}$ for every even $j< d$;
\item $\prod_{1\leq i< j}(t_i-t_j)\geq \alpha \beta^{j-2}$ for every odd $j \leq d-1$.
\end{itemize}

Therefore as in the proof of Theorem \ref{ballandquasi} we get 
$$
|E_2|\geq|\Omega|\, \min_{t} J(t)\geq c \alpha^{\frac{d}{2}} \beta^{\frac{d}{2}-1} \gamma \quad \alpha^{\frac{d}{2}-1} \beta^{\frac{{d-2}^2}{2}} \gamma^{d-1}.
$$
After substituting the values of $\alpha$, $\beta$ and $\gamma$ in terms of 
$|E_1|$, $|E_2|$ and $|F|$ we get 
$$
|E_2|\geq c \eta^{\frac{d(d-1)}{2}} |E_2|^{\frac{d}{p_0}} |E_1|^{\big(d-1\big)\big(\frac{1}{p_0}-1\big)+\frac{\big(\frac{d}{2}-1\big)\big(d-1\big)}{p_0}}
\,\,|F|^{d\big(\frac{1}{q^{'}_0}-1\big)+\frac{d-1}{q^{'}_0}+\big(\frac{d}{2}-1\big)\big(d-1\big)\big(\frac{1}{q^{'}_0}-1\big)}
$$
which implies
$$
|E_2|^{-\frac{d-1}{d+1}}\geq c \eta^{\frac{d(d-1)}{2}} |E_1|^{-\frac{d-1}{d+1}}.
$$
This is equivalent to 
$$
|E_2|\leq C \eta^{-\frac{d(d+1)}{2}} |E_1|.
$$

Now let us consider the case when $\frac{1}{p}=\frac{1}{p_1}=1-\frac{2(d-1)}{d(d+1)}$ and $\frac{1}{q}=\frac{1}{q_1}
=1-\frac{2}{d+1}$. The argument in this case is similar to the above case. Let 
$\alpha, \beta, \gamma$, $\Omega$, $\bar{y}$ and $J$ be as before.
 
Since $t_1< t_2< ...< t_d$ for every $t\in\Omega$, we have
\begin{itemize}
\item $\prod_{1\leq i< d}(t_i-t_d)\geq \gamma \alpha^{d-2} $;
\item $\prod_{1\leq i< j}(t_i-t_j) \geq \beta \alpha^{j-2}$ for every even $j< d$;
\item $\prod_{1\leq i< j}(t_i-t_j)\geq \alpha^{j-1}$ for every odd $j \leq d-1$.
\end{itemize}

Therefore we have 
\begin{multline*}
\quad\quad \quad\quad \quad |E_2|\geq|\Omega|\, \min_{t} J(t)\geq c \alpha^{\frac{d}{2}} \beta^{\frac{d}{2}-1}\gamma \quad \alpha^{d(\frac{d}{2}-1)} \beta^{\frac{d}{2}-1} \gamma\\
                           =c \eta^{\frac{d(d-1)}{2}} |E_2|^{\frac{2}{p_2}}\,\, |E_1|^{\frac{d-2}{p_2}+\big(\frac{1}{p_2}-1\big)\frac{d}{2}(d-1)}\,\, |F|^{d\big(\frac{1}{q_2^{'}}-1\big)+\frac{1}{q_2^{'}}\frac{d}{2}\big(d-1\big)}.\\
\end{multline*}
This implies
$$
|E_2|^{-\frac{d^2-3d+4}{d(d+1)}}\geq c \eta^{\frac{d(d-1)}{2}} |E_1|^{-\frac{d^2-3d+4}{d(d+1)}}.
$$
This is equivalent to 
$$
|E_2|\leq C \eta^{-\frac{d^2(d^2-1)}{d^2-3d+4}} |E_1|.
$$

We shall now consider the case when
$$
\frac{1}{p}=\frac{\theta}{p_1}+\frac{1-\theta}{p_0}\quad \text{and} \quad \frac{1}{q}=\frac{\theta}{q_1}+\frac{1-\theta}{q_0}
$$
for some $\theta\in (0,1)$ and $p_0,p_1,q_0,q_1$ as mentioned earlier. By the hypothesis of the theorem
we have
\begin{multline*}
\quad\quad\quad \quad\quad\quad T(\chi_{E_1})\geq \eta |E_1|^{\frac{\theta}{p_1}+\frac{1-\theta}{p_0}} |F|^{\frac{\theta}{{q_1}^{'}}+\frac{1-\theta}{{q_0}^{'}}-1}\\
                     =  (\eta |E_1|^{\frac{1}{p_0}} |F|^{{\frac{1}{{q_0}^{'}}}-1})^{1-\theta} (\eta |E_1|^{\frac{1}{p_1}} |F|^{{\frac{1}{{q_1}^{'}}}-1})^{\theta}.\\
\end{multline*}
and 
\begin{multline*}
\quad\quad\quad \quad\quad\quad T(\chi_{E_2})\geq \eta |E_2|^{\frac{\theta}{p_1}+\frac{1-\theta}{p_0}} |F|^{\frac{\theta}{{q_1}^{'}}+\frac{1-\theta}{{q_0}^{'}}-1}\\
                     =  (\eta |E_2|^{\frac{1}{p_0}} |F|^{{\frac{1}{{q_0}^{'}}}-1})^{1-\theta} (\eta |E_2|^{\frac{1}{p_1}} |F|^{{\frac{1}{{q_1}^{'}}}-1})^{\theta}.\\
\end{multline*}
Now let us consider the case when 
$$
|E_2|^{\frac{1}{p_0}} |F|^{{\frac{1}{{q_0}^{'}}}-1}\geq
|E_2|^{\frac{1}{p_1}} |F|^{{\frac{1}{{q_1}^{'}}}-1}.
$$
This is equivalent to 
$|E_2|^{\frac{1}{p_0}-\frac{1}{p_1}}\leq |F|^{\frac{1}{{q_1}^{'}}-\frac{1}{{q_1}^{'}}}$. Since
$|E_2|\geq |E_1|$, we also have $|E_1|^{\frac{1}{p_0}-\frac{1}{p_1}}\leq |E_2|^{\frac{1}{p_0}-\frac{1}{p_1}}
 \leq |F|^{\frac{1}{{q_1}^{'}}-\frac{1}{{q_0}^{'}}}$. This implies
 $$
|E_1|^{\frac{1}{p_0}} |F|^{{\frac{1}{{q_0}^{'}}}-1}\geq
|E_1|^{\frac{1}{p_1}} |F|^{{\frac{1}{{q_1}^{'}}}-1}.
 $$
Therefore we have  for all $x\in F$
\begin{itemize}
\item $T(\chi_{E_1})(x) \geq \eta |E_1|^{\frac{1}{p_0}} |F|^{{\frac{1}{{q_0}^{'}}}-1}$;
\item $T(\chi_{E_2})(x) \geq \eta |E_2|^{\frac{1}{p_0}} |F|^{{\frac{1}{{q_0}^{'}}}-1}$.
\end{itemize}

Now we apply the proof for the case $(p,q)=(p_0,q_0)$ to get the desired inequality.
For the other case we have 
$$
|E_2|^{\frac{1}{p_0}} |F|^{{\frac{1}{{q_0}^{'}}}-1}\leq
|E_2|^{\frac{1}{p_1}} |F|^{{\frac{1}{{q_1}^{'}}}-1}.
$$
In this case we apply the proof for $(p,q)=(p_1,q_1)$ to get the desired inequality.
\end{proof}

\smallskip

\section{Entropy refinement}
The following lemma is proved for the paraboloid in Lemma $5.3$ in \cite{MC5}. The proof for the moment curve is almost identical and so we omit the proof.
\begin{lemma}\label{entref}
Let $d\geq 2$. There exist $c, C< \infty$ with the following property. Let $\delta> 0$. 
Let $f$ be any function in $L^{p}(\mathbb{R}^d)$ satisfying
$\|Tf\|_{q}\geq (1-\delta) A \|f\|_{p}$ that has rough level set decomposition
$f=\sum_{j\in\mathbb{Z}} 2^j f_j, \chi_{E_j}\leq f_j\leq 2\chi_{E_j}$. Then for any $\eta\in(0, 1]$, 
$$
\big\| \sum_{j:2^{j} |E_{j}|^{\frac{1}{p}}< \eta \|f\|_{p}} 2^{j} f_j\big\|_{p}\leq C(\delta^{\frac{1}{p}}+\eta^c) \|f\|_{p}.
$$
\end{lemma}

\begin{lemma}\label{highint}
Let $d\geq 2$. There exist $c, C, \tilde{C} < \infty$ with the following property. Let $\rho\in(0,1)$.
Let $f$ be a $(1-\delta)$-quasiextremal for $T$. If 
$\delta\leq C\rho^{C}$, then there exists a function $\tilde{f}$ satisfying $\|f-\tilde{f}\|\leq C\rho^{C}$ 
with a rough level set decomposition $\tilde{f}=\sum_{j\in\mathbb{Z}}2^j f_{j}, f_j\sim\chi_{E_{j}}$ 
such that if both $\|2^{i}\chi_{E_i}\|\geq \rho$ and $\|2^{j}\chi_{E_j}\|\geq \rho$, then 
$$
|i-j|\leq \tilde{C}\rho^{-\tilde{C}}.
$$
\end{lemma}
\begin{proof}
This lemma is an improvement over the previous lemma, in the sense that the indices
$\{j\}$ in the sum for which $\|2^{j}\chi_{E_{j}}\|_{p}\geq \eta$, can not be too far from each other.
All of them are inside an interval of $\mathbb{Z}$ of length at most $C\,\eta^{-C}$. 
This lemma is proved corresponding to the paraboloid in Lemma $6.1$ in \cite{MC5}. The proof 
of this lemma will be an application of Lemma \ref{extrapolation} together with the previous lemma.The proof
is almost identical in our case. Therefore we omit the proof.
\end{proof}

Following Corollary $6.3$ in \cite{MC5} the above lemma immediately implies the following.
\begin{corollary}\label{decay}
There exist a finite constant $C$ and a function $\Psi:(0,\infty)\rightarrow (0,\infty)$
satisfying $\frac{\Psi(t)}{t^p}\rightarrow 0$ as $t\rightarrow 0,\infty$ with the following property.
For any $\epsilon> 0$ there exists a $\delta> 0$ such that for any nonnegative 
function $f$ with $\|f\|_{p}=1$ and $\|T(f)\|_{q}\geq (1-\delta)A$, there exists $\phi\in G_{d}$
and a decomposition $\phi^{*}(f)=g+h$ with $g,h \geq 0$ satisfying $\|h\|_{p}<\epsilon$ and 
\begin{equation}
\int \Psi(g)\leq C.
\end{equation}
\end{corollary}

\smallskip

\section{Uniform decay and Extremizers at non-end points}

In this section we show that any extremizing sequence behaves in a uniform manner. Using this we prove that at non end points
any extremizing sequence, after applying symmetries if necessary, converges to a non zero function in $L^{p}$. By continuity, the limit must be an
extremizer for the corresponding $L^{p}$ bound.

\subsection{Spatial Localization}
\begin{lemma}
There exists $C< \infty$ such that for any $\epsilon> 0$ there exists $\delta> 0$ with the
following property. Let $f$ be a nonnegative function with $\|f\|_{p}=1$ and $\|T(f)\|_{q}\geq (1-\delta) A$.
Then there exists $F$ with rough level set decomposition $F=\sum_{j\in S}2^{j}f_j, f_j\sim \chi_{E_j}$ satisfying
$$
0\leq F\leq f,
$$
$$
\|T(F)\|_{q}\geq (1-\epsilon) A,
$$
$$
|i-j|\leq  C \epsilon^{-C} \quad\text{for all}\quad i, j\in S
$$ 
and for each $j\in S$ there exist $N(\sim C\epsilon^{-C})$ paraballs $B_{j,i}$ such that
$$
E_{j}\subset \bigcup_{i=1}^N B_{j,i};
$$
$$
\sum_{i=1}^N |B_{j,i}|\leq C \epsilon^{-C} |E_{j}|.
$$
\end{lemma}

\begin{proof} 
Let $\epsilon>0$. Fix  $\delta>0$ sufficiently small for later purposes, and suppose
$\|f\|_p=1$, $\|Tf\|_{q}\geq (1-\delta)A$. By using Lemma \ref{highint}, by losing $\epsilon$ amount of
$L^{p}$ norm we can assume that $f$ has a finite level set decomposition. 
In other words, $f=\sum_{S} 2^{j}f_j$ with $S\subset (J-C\epsilon^{-C}, J+C\epsilon^{-C})\cap\mathbb{Z}$ for some $J\in\mathbb{Z}$.
Let $0< \eta\,(\leq c\epsilon^C)$ be a small quantity to be chosen later. Then $\|T(f)\|_{q}\geq (1-2\delta)\,A\geq\eta$.
Now we apply Lemma \ref{quasi} to $f$ to get a paraball $B_{1}$ and $i_{1}\in S$ such that
$\|2^{i_{1}}\chi_{E_{i_{1}}\cap B_{1}}\|_{p}\geq c\eta^{C}$ and $|B_{1}|\leq|E_{i_{1}}|$.

At the next step we  set $g_{1}=f\chi_{E_{i_{1}}\cap B_{1}}$ and write $f=g_{1}+h_{1}$. So
$h_{1}=\sum_{j\neq _{i_{1}}}f(\chi_{E_{j}}+\chi_{E_{i_{1}}\setminus B_{1}})$. Now we look
at $\|T(h_{1})\|_{q}$. If $\|T(h_{1})\|_{q}\geq \eta$ then by applying Lemma \ref{quasi} to $h_{1}$ we get 
another paraball $B_{2}$ and $i_{2}\in\mathbb{Z}$ such that
$$
\|2^{i_{2}}\chi_{E_{i_{2}}\cap B_{2}\setminus\big(E_{i_{1}}\cap B_{1}\big)}\|\geq c\eta^{C}\quad 
\text{and}\quad |B_{2}|\leq|E_{i_{2}}|.
$$
Now define $g_{2}=f\chi_{E_{i_{2}}\cap B_{2}\setminus\big(E_{i_{1}}\cap B_{1}\big)}$ 
and $f=g_{1}+g_{2}+h_{2}$.

We continue this process. Now suppose we are at the $(n-1)$-th step. So we have
a collection of paraballs $\{B_{j}\}_{1\leq j\leq n-1}$ and indices $\{i_{j}\}_{1\leq j\leq n-1}$
such that
\begin{itemize}
\item $|B_{j}|\leq |E_{i_{j}}|$;
\item $g_{m}=f\chi_{E_{i_{m}}\cap B_{m}\setminus\cup_{1\leq j\leq m-1}(E_{i_{j}}\cap B_{j})}$;
\item $\|g_{m}\|_{p}\geq c\,\eta^{C}$;
\item $f=\sum_{j=1}^{n-1} g_{j}+h_{n-1}$.
\end{itemize}
Now if $\|T(h_{n-1})\|_{q}<\, \eta$ we stop. Otherwise after applying Lemma \ref{quasi}
one more time we get another paraball $B_{n}$ and $i_{n}\in\mathbb{Z}$ such that 
$\|2^{i_{n}}\chi_{E_{i_{n}}\cap B_{n}}\|_{p}\geq c\eta^{C}$ and $|B_{n}|\leq |E_{i_{n}}|$.

Since the $g_{j}$ have disjoint support, $\sum_{1\leq j\leq n} \|g_{j}\|_{p}^{p}\leq \|f\|_{p}^{p}\leq 1$.
So this process must stop after at most $C\eta^{-C}$ steps. Let the process stops at the $n$-th
step. Then we define $F=\sum_{1\leq j\leq n} g_{j}$, so that $\|T(f-F)\|_{q}< \,\eta$.
This means $\|T(F)\|_{q}\geq (1-2\delta-\eta) A\geq (1-\epsilon) A$ provided
$\eta$ is sufficiently small compared to $\epsilon$. At the same time since $\|T(F)\|_{q}\leq A\|F\|_{p}$,
we have $\|F\|_{p}\geq 1-\epsilon$. This implies $\|f-F\|_{p}\leq \epsilon$. Now we set
$\{B_{j,i}\}=\{B_{i}: i_{j}=j\}$.
\end{proof}

\begin{lemma}\label{spatial}
There exists $C< \infty$ such that for any $\epsilon> 0$ there exists 
$\delta> 0$ with the following property. Let $f$ be a nonnegative function 
with $\|f\|_{p}=1$ and $\|T(f)\|_{q}\geq (1-\delta)\,A$. Then there exists $\tilde{f}$ 
with rough level set decomposition $\tilde{f}=\sum_{j\in S}2^{j}f_j$, $f_j \sim\chi_{E_j}$ satisfying
$$
0\leq \tilde{f}\leq f,
$$
$$
\|\tilde{f}\|_{p}\geq (1-\epsilon),
$$
$$
\|T(\tilde{f})\|_{q}\geq (1-\epsilon)\,A
$$
and there exists a distinguished $J\in S$ and a paraball $B_J$ such that 
$$
|J-j|\leq C \epsilon^{-C}\quad\text{for all}\quad j\in S,
$$ 
$$
E_{j}\subset C\epsilon^{-C^{C\epsilon^{-C}}} B_J\quad\text{for all}\quad j\in S,
$$
$$
\|2^{J}\chi_{{B}_J}\|_{p}\leq C.
$$
\end{lemma}
\begin{proof}
This lemma is an improvement over the previous lemma, in the sense that the paraballs 
$\{B_{j,i}\}$ have been replaced by a single paraball $B$, after scaling it with a factor of
$(C\epsilon^{-C})^{C\epsilon^{-C}}$. The proof of this lemma will be an application of Lemma \ref{orthogonal}
together with the previous lemma.

By the previous lemma we can assume that $f=\sum_{j\in S}2^{j}f_j, f_j\sim\chi_{E_{j}}$ where
$E_{j}\subset \cup_{i=1}^{C\epsilon^{-C}} B_{j,i}$ and $|S|\leq C\epsilon^{-C}$. Let 
$0<\,\eta<\, \epsilon^{C}$ be a small quantity to be chosen later. Let us write the collection of paraballs
$B_{j,i}$ as $\{B_{l}:1\leq l\leq N\}$. Then $N\leq C\epsilon^{-C}$. Let if possible $\{1,2,...,N\}=S^{a}\cup S^{b}$ 
be a partition of $\{1,2,...,N\}$ such that for each $(i,j)\in S^a\times S^b$,
$d(B_i, B_j)> C \eta^{-C}$. We continue as in the proof of Lemma \ref{highint}.  Let 
$g=\sum_{k\in\tilde{S}} 2^{k}g_k, g_k\sim\chi_{ F_k}$, be an arbitrary $L^{q^{'}}$ function with $\|g\|_{q^{'}}=1$. 
Now  partition each of the sets $F_{k}$ measurably 
as $F_{k}=F_{k}^{a}\cup F_{k}^{b}\cup F_{k}^{c}$ with the following property:
\begin{itemize}
\item For each $x\in F_{k}^{a}$ there exists $j\in S^{a}$ so that $T\chi_{B_{j}}(x)\geq \eta 
|B_{j}|^{\frac{1}{p}}\, |F_{k}|^{-\frac{1}{q}}$;
\item For each $x\in F_{k}^{b}$ there exists $j\in S^{b}$ so that $T\chi_{B_{j}}(x)\geq \eta 
|B_{j}|^{\frac{1}{p}}\, |F_{k}|^{-\frac{1}{q}}$.
\end{itemize}

Write 
$$
h^{a}=g\sum_{k\in\tilde{S}}2^{k}\chi_{F_{k}^{a}},\qquad\qquad h^{b}=g\sum_{k\in\tilde{S}}2^{k}\chi_{F_{k}^{b}},
$$
$$
f^{a}=f\sum_{j\in S^{a}}2^{j}\chi_{E_{j}},\qquad\qquad f^{b}=f\sum_{j\in S^{b}}2^{j}\chi_{E_{j}},
$$
As in the proof of Lemma \ref{highint}, if $\eta$ is sufficiently small, there is $i\in S^{a}$ and
$k\in \tilde{S}$ with
$$
T(\chi_{B_{i}}, \chi_{F_{k}^{b}})\geq \eta |B_{i}|^{\frac{1}{p}} |F_{k}|^{1-\frac{1}{q}}.
$$

But by the proof of Lemma \ref{orthogonal} this implies there is $j\in S^b$, such that
$d(B_i, B_j)\leq C\epsilon^{-C}$ which contradicts our hypothesis. Therefore it is not possible to
decompose the collection of paraballs corresponding to $f$ into a disjoint union of two sets such that
any two elements belonging to different sets are at least $C\epsilon^{-C}$ far with respect to the 
mock-distance $d$. Now let us fix a paraball $B_{j_0}$ corresponding to $f_0$. Now we construct
inductively a sequence of collection of paraballs by
\begin{itemize}
\item $\{B_1,....,B_N\}=\cup_{j=0}^N \mathbb{B}_j$;
\item $\mathbb{B}_0=\{B_{j_0}\}$;
\item $B\in \mathbb{B}_j$ and $B'\in \mathbb{B}_{j+1}$ implies $d(B, B')\leq C\epsilon^{-C}$.
\end{itemize}
By quasi-tirangle inequality this implies $d(B_{j_0}, B_j)\leq C\epsilon^{-C^{C\epsilon^{-C}}}$for all $j=1,2,...,N$.
\end{proof}
\smallskip

\section{Weak Convergence and Extremizers for $\theta\in(0,1)$}\label{weak}
\begin{lemma}\label{localization}
There exists a constant $C$ (depending only on $d$) and positive functions 
$\Psi_{1}, \Psi_{2}:(0,\infty)\rightarrow (0,\infty)$ and 
$\rho_{1}, \rho_{2}:(1,\infty)\rightarrow (0,\infty)$ satisfying 
$\frac{\Psi_{1}(t)}{t^p}\rightarrow 0$ and $\frac{\Psi_{2}(t)}{t^{q'}}\rightarrow 0$ as $t\rightarrow 0,\infty$ and $\rho_{i}(R)\rightarrow\infty$ 
as $R\rightarrow\infty$ with the following property. For any $\epsilon\,>\,0$ 
there exists a $\delta> 0$ such that for any nonnegative 
function $f$ with $\|f\|_{p}=1$ and $\|T(f)\|_{q}\geq (1-\delta)\,A$, there exists
$\phi\in G_{d}$ and a decomposition
$$
\phi^{*}(f)=g+h
$$
with $g, h\geq 0$ satisfying for large $R$,
$$
\|h\|_{p}< \epsilon,\quad \int_{\mathbb{R}^{d}} \Psi_{1}(g)\leq C, 
\quad \text{support}(g)\subset B(0,\rho_{1}(R)).
$$
In addition, there exists $F\geq 0$ satisfying $\|F\|_{q^{'}}=1$ and 
$$
\langle T(g), F\rangle\geq (1-\epsilon)\,A,\quad \int_{\mathbb{R}^{d}} \Psi_{2}(F)\leq C, 
\quad \text{support}(F)\subset B(0,\rho_{2}(R)).
$$
\end{lemma}
\begin{proof}
Let $\epsilon>0$. Fix $\delta>0$ sufficiently small. Let $\|f\|_{p}=1$ with $\|Tf\|_{q}\geq A(1-\delta)$.
By applying Lemma \ref{spatial} to $f$ we can assume that there is a $\phi\in G_d$ such that
$$
\phi^{*}(f)=g+h
$$
with $\|h\|_{p}<\epsilon$ and support$(g)\subset B_J=B(\bar x_J,t_J,\alpha_J,\beta_J)$. In addition, by Lemma \ref{compa},
we can assume that in Lemma \ref{spatial} the distinguished index $J=0$ and the paraball $B_J=\{y\in\mathbb{R}^{d}; \|y\|<1\}$.
Therefore $\|h\|_{p}<\epsilon$ and support($g$)$\subset B(0,C\epsilon^{-C^{C\epsilon^{-C}}})$.

We set $\rho(R)=CR^{C^{CR^{C}}}$. The second part of the conclusion follows similarly
by applying the same proof to the operator $T^{*}$ since $T^{*}$ is the convolution with
the affine arc length measure on the curve $-h(t)$.
\end{proof}

\subsection{Proof of existence of extremizers for $\theta\in(0,1)$}\label{existence}
The proof of the existence of extremizers is
along the lines of the proof for the corresponding result in \cite{MC5}. For the sake of completeness we give the details of the proof.
Let $\{f_n\}$ be any extremizing sequence.
By the previous Lemma there exist $\{\phi_{n}\in G_d\}$, such that $\phi^{*}_n(f_{n})=g_{n}+h_{n}$, while the
functions $g_{n}$ and $h_{n}$ satisfy all the conclusions of the previous Lemma corresponding to
$\epsilon_n=\frac{1}{n}$. Also there exists a sequence $\{F_{n}\}$ with $\|F_{n}\|_{q^{'}}=1$ and 
$T(g_{n}, F_{n})\rightarrow A$.

By applying the Banach-Alaoglu Theorem, after passing through a subsequence  we can assume 
that $g_{n}\rightharpoonup g$ and $F_{n}\rightharpoonup F$ weakly. Therefore
$\|F\|_{{q}^{'}}\leq 1$ and $\|g\|_{p}\leq 1$. Let us fix a large $R$. Now we set
\begin{align}
g_{n,R}(x)&=g_{n}(x)\,\chi_{\|x\|\leq R}(x)\,\chi_{g_{n}\leq R}(x),\qquad g_{R}(x)=g(x)\,
\chi_{\|x\|\leq R}(x)\,\chi_{g\leq R}(x);\\
F_{n,R}(x)&=F_{n}(x)\,\chi_{\|x\|\leq R}(x)\,\chi_{F_{n}\leq R}(x),.\qquad F_{R}(x)=F(x)\,
\chi_{\|x\|\leq R}(x)\,\chi_{F\leq R}(x).
\end{align}
By stationary phase argument for any fixed $\psi\in C_{0}^{1}(\mathbb{R}^{d})$, the operator $f\mapsto \psi T(\psi f)$
maps $L^{2}(\mathbb{R}^{d})$ boundedly to the Sobolev space $H^{\frac{1}{d}}$, which
in turn embeds into $L^{s}$ where $\frac{1}{s} = \frac{1}{2}-\frac{1}{d^{2}}$.
Thus the weak convergence of $g_{n, R}$ to $g_{R}$ as $n\rightarrow\infty$ implies 
the $L^{s}$ norm convergence of $T(g_{n, R})$ to $T(g_{R})$ as $n\rightarrow\infty$, for 
every fixed $R$. Therefore $T(G_{n, R}, F_{n, R})\rightarrow T(g_{R}, F_{R})$ as 
$n\rightarrow \infty$ for a fixed $R$.

By Lemma \ref{localization} the integral of $g^{p_{\theta}}_{n}$ outside the ball of
radius $R$ goes to zero as $R$ goes to infinity uniformly in $n$. This implies
$g_{n,R}$ converges to $g_{n}$ in $L^{p_{\theta}}$ as $R$ goes to infinity uniformly in $n$. Similarly
$F_{n,R}$ converges to $F_{n}$ in $L^{q^{'}_{\theta}}$ uniformly in $n$. This together with the 
conclusion from previous paragraph implies  
\begin{equation}\label{limit}
A=\lim_{n\rightarrow\infty}T(g_{n}, F_{n})=T(g, F).
\end{equation}
So $g$ is an extremizer.

\smallskip

\section{$L^{p}$ convergence of extremizing subsequence}
The main result of this section is the $L^{p}$ convergence of a subsequence of any extremizing
sequence after applying suitable symmetries. The proof is similar to the corresponding
result in ~\cite{MC5}. For the convenience of the reader we give the details of the proof.

\textit{\bf{Euler-Lagrange identity}}: Let $f$ be a nonnegative extremizer with $\|f\|_{p}=1$.
Then by Holder's inequality
\begin{multline*}
A^{q}=\|T(f)\|_{q}^{q}=\langle T(f), T(f)^{q-1}\rangle=\langle f, T^{*}(T(f)^{q-1})\rangle\\
\qquad\qquad \qquad\qquad\leq \|f\|_{p} \|T^{*}(T(f)^{q-1})\|_{p^{'}}= A \|(Tf)^{q-1}\|_{q{'}}=
A \|T(f)\|_{q}^{\frac{q}{q^{'}}}=A A^{\frac{q}{q^{'}}}=A^{q}.\\
\end{multline*}
Since the equality holds for the above chain of inequalities, we have $T^{*}(T(f)^{q-1})$ agrees 
with a constant multiple of $f^{\frac{p}{p^{'}}}$ almost everywhere on $\mathbb{R}^{d}$. The 
above equality implies this constant is $A^{q}$. So finally, we have for any nonnegative 
extremizer $f$ with $\|f\|_{p}=1$,
\begin{equation}\label{EL}
T^{*}\big((T(f))^{q-1}\big)=A^{q} f^{p-1}\quad\text{almost everywere on}\quad\mathbb{R}^{d}.
\end{equation}

\begin{lemma}\label{strongconv}
Let $f_{n}$ be an extremizing sequence. Then there exist a sequence of symmetries
$\{\phi_{n}\}\subset G_{d}$ and an extremal $F$ such that $\{\phi_{n_{k}}^{*} (f_{n_{k}})\}$
converges to $F$ in $L^{p}$ for some subsequence $\{f_{n_{k}}\}$.
\end{lemma}
\begin{proof}
After passing through a subsequence, if necessary and applying suitable symmetry
$\phi_{n}^{*}$ we have that $\phi_{n}^{*}(f_{n})=g_{n}+h_{n}$ with $\|h_{n}\|_{p}\rightarrow 0$
and $\|T(g_{n})\|_{q}\rightarrow A$. Therefore it is enough to prove that $\{\phi_{n}^{*}(f_{n})\}$ converges
to $g$ in $L^{p}$ where $g$ is as in (\ref{limit}).

By (\ref{EL}) there exists $H\geq 0$ such that $T^{*}(H)=A^{q} g^{p-1}$ almost everywhere
on $\mathbb{R}^{d}$. By (\ref{limit}) we have
\begin{multline*}
A^{q}=A^{q}\langle g, g^{p-1}\rangle=\langle g, T^{*}(H)\rangle=\lim_{n\rightarrow\infty}\langle g_{n},
T^{*}(H)\rangle\\
=A^{q}\lim_{n\rightarrow\infty}\langle g_{n}, g^{p-1}\rangle=A^{q}\lim_{n\rightarrow\infty}\langle \phi_{n}^{*}(f_{n}),g^{p-1}\rangle.\\
\end{multline*}
At this point we apply Theorem $2.11$ in \cite{LL} to the pair $(\{\phi_{n}^{*}(f_{n})\}, g)$ to get the desired
conclusion.
\end{proof}
\smallskip

\section{Extremizers at end points}
In this section we shall prove Theorem \ref{MT2}, which describes a relation between extremizers for $T:L^{p_{0}}\rightarrow L^{q_{0}}$ and extremizers for $X^{*}:L^{p_{0}}\rightarrow L^{q_{0}}$, adjoint of $X$ as defined in \ref{defX}. Simultaneously we prove a relation between extremizers for $T:L^{p_{1}}\rightarrow L^{q_{1}}$ and extremizers for $X:L^{p_{1}}\rightarrow L^{q_{1}}$. We would like to thank Michael Christ for his suggestion to look at the restricted X-ray transform, $X$, for the endpoint cases.

\begin{lemma}\label{norms}
Let $T$ and $X$ be defined as in \ref{defT} and \ref{defX} respectively and $d>2$. Then $\|T\|_{L^{p_0}\rightarrow L^{q_0}}\geq\|X^{*}\|_{L^{p_0}\rightarrow L^{q_0}}$ and $\|T\|_{L^{p_1}\rightarrow L^{q_1}}\geq\|X\|_{L^{p_1}\rightarrow L^{q_1}}$.
\end{lemma}
\begin{proof}
It suffices to show that $\|T\|_{L^{p_0}\rightarrow L^{q_0}}\geq\|X^{*}\|_{L^{p_0}\rightarrow L^{q_0}}$ since $(p_1,q_1)=(q'_0,p'_0)$ and $T^{*}$ is the same operator as $T$ with the curve $h(t)$ replaced by $-h(t)$. Let $\epsilon>0$. Let $f\in L^{p_0}$ and $g\in L^{q'_0}$. Let $\gamma(t)=(t^2,t^3,...,t^d)\in\mathbb{R}^{d-1}$. Define
$$
\tilde{f}_\epsilon(x):={\epsilon}^{\frac{d-1}{p_0}} f((0,\epsilon(x_2,...,x_d)+h(x_1))
$$ 
and
$$
\tilde{g}_\epsilon(x)={\epsilon}^{\frac{d}{q^{'}_0}}g(\epsilon x).
$$

Let $f$ and $g$ be compactly supported smooth functions. Then
\begin{equation}
\begin{split}
\langle X^{*}f,g\rangle&=\lim_{\epsilon\to 0}\int_{(t,y)\in\mathbb{R}\times\mathbb{R}^{d}}f\big(\epsilon y_1+t, y-\frac{\gamma(\epsilon y_1+t)-\gamma(t)}{\epsilon}\big)g(y)\,dy\,dt\\
&=\lim_{\epsilon\to 0}\langle Tf_{\epsilon}, g_{\epsilon}\rangle\leq \|T\|
\end{split}
\end{equation}
where $f_{\epsilon}, g_\epsilon$ are the functions such that $\tilde{f_{\epsilon}}=f$ and $\tilde{g_{\epsilon}}=g$.
\end{proof}

\begin{lemma}\label{samenorm}
Let $T$ and $X$ be as above and $d>2$.
\begin{itemize}
\item If there exists an extremizing sequence for $T:L^{p_0}\rightarrow L^{q_0}$ that does not have a subsequence converging to an extremizer modulo symmetries of $T$, then $\|T\|_{L^{p_0}\rightarrow L^{q_0}}=\|X^{*}\|_{L^{p_0}\rightarrow L^{q_0}}$.
\item  If there exists an extremizing sequence for $T:L^{p_1}\rightarrow L^{q_1}$ that does not have a subsequence converging to an extremizer modulo symmetries of $T$, then $\|T\|_{L^{p_1}\rightarrow L^{q_1}}=\|X\|_{L^{p_1}\rightarrow L^{q_1}}$.
\end{itemize}
\end{lemma}
\begin{proof}
We shall prove the lemma only for $T:L^{p_0}\rightarrow L^{q_0}$, the other case being identical. By the previous Lemma it suffices to show that 
$\|T\|_{L^{p_0}\rightarrow L^{q_0}}\leq\|X^{*}\|_{L^{p_0}\rightarrow L^{q_0}}$. By hypothesis, there exists an extremizing sequence $\{f_n\}$ for $T:L^{p_0}\rightarrow L^{q_0}$ such that for any sequence of symmetries $\{\phi_{n}^*\}$, the sequence $\{\phi_{n}^*(f_n)\}$ has no subsequence which converges to a non-zero limit in $L^{p_0}$. Let us start with such an extremizing sequence $\{f_n\}$ such that $\|Tf_n\|_{L^{q_0}}\geq(1-\frac{1}{n}) A$ for each $n$. Then there exists another sequence $\{g_n\}$ such that $\langle Tf_n, g_n\rangle$ converges to $A=\|T\|_{L^{p_0}\rightarrow L^{q_0}}$ and $\|g_n\|_{L^{q'_0}}=1$. By Lemma \ref{spatial} there is a sequence of symmetries $\{\phi_{n}^*\}$ such that after changing $f_n$ to $\phi_{n}^*(f_n)$ and $g_n$ to $\psi_{n}^*(g_n)$, if necessary, we have
\begin{itemize}
\item $f_n=\sum_{j\in S_n}2^{j}f_{n,j}$, $f_{n,j} \sim\chi_{E_{n,j}}$;
\item $|j-J_n|\leq C n^{-C}\quad\text{for all}\quad j\in S_n$;
\item $E_{n,j}\subset Cn^{{-C}^{Cn^{-C}}}B (0,0,\alpha_n,1)\quad\text{for all}\quad j\in S_n$;
\item $\|2^{J_n}\chi_{B (0,0,\alpha_n,1)}\|_{p_0}\leq C$.
\end{itemize}
and
\begin{itemize}
\item $g_n=\sum_{k\in \tilde{S}_n}2^{k}g_{n,k}$, $g_{n,k}\sim\chi_{F_{n,k}}$;
\item $|k-K_n|\leq C n^{-C}\quad\text{for all}\quad k\in \tilde{S}_n$;
\item $F_{n,k}\subset Cn^{{-C}^{Cn^{-C}}}B^{*} (0,0,\alpha_n,1)\quad\text{for all}\quad k\in \tilde{S}_n$;
\item $\|2^{K_n}\chi_{B^{*} (0,0,\alpha_n,1)}\|_{{q'_0}}\leq C$.
\end{itemize}
Since $f_n$ has no subsequence converging to a nonzero limit in $L^{p_0}$, by the proof of Lemma \ref{localization}, $\alpha_n\rightarrow 0$.

We define 
$$
\tilde{f}_{n}(x):={\alpha_n}^{\frac{d-1}{p_0}} f_{n}((0,\alpha_{n}(x_2,...,x_d)+h(x_1))
$$ 
and
$$
\tilde{g}_n(x)={\alpha_n}^{\frac{d}{q^{'}_0}}g_{n}(\alpha_{n}x).
$$

Let $c_n=Cn^{{-C}^{Cn^{-C}}}$. This implies for each $\tilde{f}_n$, we now have the following,
$$
\tilde{f}_n=\sum_{|j|<c_n}2^{j}f_{n,j}((0,2^{\frac{-p_{0}J_{n}}{d-1}}(x_2,...,x_d)+h(x_1))).
$$
Let $f^{*}_{n,k}(x)= f_{n,k}((0,2^{\frac{-p_{0}J_{n}}{d-1}}(x_2,...,x_d)+h(x_1)))$. Then $f^{*}_{n,j}\sim\chi_{E^{*}_{n,j}}$ with the following properties:
\begin{itemize}
\item $\chi_{E^{*}_{n,0}}\sim\chi _{B(0,1)}$;
\item $E^{*}_{n,j}\subset B(0,c_n)$ for all $|j|<c_n$.
\end{itemize}
Similarly we have
$$
\tilde g_n(x)=\sum_{|k|<c_n}2^kg_{n,k}(2^{\frac{-q'_{0}K_{n}}{d}}x).
$$
If $g^{*}_{n,k}(x)=g_{n,k}(2^{\frac{-q'_{0}K_{n}}{d}}x)$ then $g^{*}_{n,k}\sim\chi_{F^{*}_{n,k}}$ with
\begin{itemize}
\item $\chi_{F^{*}_{n,0}}\sim\chi _{B(0,1)}$;
\item $F^{*}_{n,k}\subset B(0,c_n)$ for all $|j|<c_n$.
\end{itemize}

By a simple change of variable we see that for each $n$, we have $\|\tilde{g}_{n}\|_{L^{q'_0}}=1$ and $\|\tilde{f}_{n}\|_{L^{p_0}}=1$. Furthermore, we now apply the proof of Lemma \ref{localization} to show that there is an $f\in L^{p_0}$ such that $\{\tilde{f}_n\}$ has a subsequence that converges weakly to $f$ as $n$ goes to infinity and there is an $g\in L^{q'_0}$ such that $\{\tilde{g}_n\}$ has a subsequence weakly converges to $g$ as $n$ goes to infinity. WLOG we assume that the sequence $\{\tilde{f}_n\}$ it self converges weakly to $f$ and likewise for $\{\tilde{g}_n\}$. We shall now prove that 
\begin{equation}
\begin{split}
\|T\|_{L^{p_0}\rightarrow L^{q_0}}&=\lim_{n\rightarrow\infty}\langle Tf_n, g_n\rangle\\
&=\lim_{n\rightarrow\infty}\int_{(t,y)\in\mathbb{R}\times\mathbb{R}^{d}}\tilde{f_n}(\epsilon y_1+t, y-\frac{\gamma(\epsilon y_1+t)-\gamma(t)}{\epsilon})\tilde{g_n}(y) \,dy\,dt\\
&=\lim_{n\rightarrow\infty}\langle X^{*}(\tilde f_n), \tilde g_n\rangle\\
&= \langle X^{*}f,g\rangle\\
&\leq \|X^{*}\|_{L^{p_0}\rightarrow L^{q_0}}.
\end{split}
\end{equation}
where $\epsilon=\alpha_n$.

The proof of this claim is similar to the proof of \ref{existence}. We define for a large $R$,
\begin{align}
\tilde f_{n,R}(x)&=\tilde f_{n}(x)\,\chi_{\|x\|\leq R}(x)\,\chi_{\tilde f_{n}\leq R}(x),\qquad f_{R}(x)=f(x)\,
\chi_{\|x\|\leq R}(x)\,\chi_{f\leq R}(x);\\
\tilde g_{n,R}(x)&=\tilde g_{n}(x)\,\chi_{\|x\|\leq R}(x)\,\chi_{\tilde g_{n}\leq R}(x),.\qquad g_{R}(x)=g(x)\,
\chi_{\|x\|\leq R}(x)\,\chi_{g\leq R}(x).
\end{align}

For any compactly supported smooth function $\psi$, the operator $f\mapsto \psi X^{*}(\psi f)$
maps $L^{2}(\mathbb{R}^{d})$ boundedly to the Sobolev space $H^{s}$, which
in turn embeds into $L^{q}$ where $q>2$, see \cite{CNSW}.
Thus the weak convergence of $\tilde f_{n, R}$ to $f_{R}$ as $n\rightarrow\infty$ implies 
the $L^{q}$ norm convergence of $X^{*}(\tilde f_{n, R})$ to $X^{*}(f_{R})$ as $n\rightarrow\infty$, for 
every fixed $R$. Therefore $\langle X^{*}\tilde f_{n, R}, \tilde g_{n, R}\rangle\rightarrow \langle X^{*}f_{R},g_{R}\rangle$ as 
$n\rightarrow \infty$ for a fixed $R$.

By Lemma \ref{localization}  we know that the $L^{q^{'}_0}$ norms of $\tilde{g}_n$ and the $L^{p_0}$ norms of $\tilde{f}_n$ decreases uniformly in $n$ outside the ball of radius $R$ centered at $0$ as $R$ goes to infinity. This together with the previous paragraph imply that $\|T\|\leq \|X^{*}\|$.
\end{proof}

We now have the following immediate corollary.
\begin{corollary}\label{weak}
Let $T$ and $X$ be as above and $d>2$.
\begin{itemize}
\item If there exists an extremizing sequence, $\{f_n\}$, for $T:L^{p_0}\rightarrow L^{q_0}$ that does not have a subsequence converging to an extremizer modulo symmetries of $T$, then after applying the nonsymmetry, $f_n\rightarrow r^{\frac{d-1}{p_0}}_nf_n((0,r_{n}x')+h(x_1))$, it has a subsequence that converges weakly to an extremizer for $X^*:L^{p_0}\rightarrow L^{q_0}$.
\item  If there exists an extremizing sequence, $\{f_n\}$, for $T:L^{p_1}\rightarrow L^{q_1}$ that does not have a subsequence converging to an extremizer modulo symmetries of $T$, then after applying the nonsymmetry, $f_n\rightarrow r^{\frac{d}{q'_0}}_nf_n(r_{n}x)$, it has a subsequence that converges weakly to an extremizer for $X:L^{p_1}\rightarrow L^{q_1}$.
\end{itemize}
\end{corollary}

\subsection{Proof of Theorem \ref{MT2}}
To complete the proof of Theorem \ref{MT2}, it suffices to prove that the weak convergence in Corollary \ref{weak} is in fact $L^{p}$ convergence.
The proof is along the same line of proof of Lemma \ref{strongconv}. Let $\{\tilde{f}_n\}$ converges weakly to $f$ in $L^{p_0}$ and $\{\tilde{g}_n\}$
converges weakly to $g$ in $L^{q_0}$. Since $g$ is an extremizer for $X:L^{q^{'}_{0}}\rightarrow L^{p^{'}_{0}}$, By the Euler Lagrange equation for $X$ we have
$$
X^{*}\big((X(g))^{p^{'}_{0}-1}\big)= A^{p^{'}_{0}} g^{q^{'}_{0}-1}.
$$

Now we apply Theorem $2.11$ in \cite{LL} to the tuple $(p,g_n,g)=(q^{'}_{0}, \tilde{g}_{n},g)$ to prove that $\tilde{g}_{n}$ converges to $g$ in $L^{q^{'}_{0}}$. Similarly $\tilde{f}_{n}$ converges to $f$ in $L^{p_{0}}$.
\smallskip

\noindent{\textbf{Acknowledgements.} I would like to thank my Ph.D. adviser, Betsy Stovall, for suggesting this problem and for many insightful comments and numerous discussions during the course of this work. This work would have been impossible without her guidance. I also like to thank Almut Burchard and Michael Christ for their many valuable remarks, in particular Christ's suggestion to look at the X-ray transform for the end point case. I would like to thank Andreas Seeger for pointing out few references and many suggestions that improved the exposition of this article. This work was supported in part by NSF grants DMS-1266336 and DMS-1600458.  Part of this work was carried out when the author visited the Mathematical Sciences Research Institute as a program associate in the Harmonic Analysis program in Spring 2017, supported by DMS-1440140.  A version of this article formed part of the author's Ph.D. thesis. I would like to thank the referees for their comments and suggestions which significantly improved the exposition of this article.
\smallskip

\bibliographystyle{plain}
%bibliographystyle{abbrvnat}
\bibliography{18054}
\end{document}